\documentclass{article}

\usepackage[utf8]{inputenc}
\usepackage{geometry, mathtools, amsmath, amsfonts, amssymb, amsthm, cases, thmtools, tabularx, appendix, xcolor, bbm, titling, bm, enumitem}
\usepackage{tikz,lipsum,lmodern}
\usepackage[most]{tcolorbox}

\newtheorem{theorem}{Theorem}[section]
\newtheorem{definition}[theorem]{Definition}
\newtheorem{proposition}[theorem]{Proposition}
\newtheorem{lemma}[theorem]{Lemma}
\newtheorem{corollary}[theorem]{Corollary}
\newtheorem{remark}[theorem]{Remark}

\numberwithin{equation}{section} 

\newcommand{\eps}{\varepsilon}
\newcommand{\R}{\mathbb{R}}
\usepackage[colorlinks=true, pdfstartview=FitV, linkcolor=blue,citecolor=blue, urlcolor=blue]{hyperref}

\DeclareMathOperator*{\esssinf}{ess\,inf}
\DeclareMathOperator*{\esssup}{ess\,sup}
\title{Duality solutions to the hard-congestion model for the dissipative Aw-Rascle system}
\author{Nilasis Chaudhuri\thanks{University of Warsaw, Poland; nchaudhuri@mimuw.edu.pl}, Muhammed Ali Mehmood\thanks{Imperial College London, London, United Kingdom; muhammed.mehmood21@imperial.ac.uk}, Charlotte Perrin\thanks{Aix Marseille Univ, CNRS, I2M, Marseille, France; charlotte.perrin@cnrs.fr}, Ewelina Zatorska\thanks{University of Warwick, United Kingdom; ewelina.zatorska@warwick.ac.uk}}
\date{\today}
\begin{document}
\maketitle
\thispagestyle{empty}

\begin{abstract}
We introduce the notion of duality solution for the hard-congestion model on the real line, and additionally prove an existence result for this class of solutions. Our study revolves around the analysis of a generalised Aw-Rascle system, where the offset function is replaced by the gradient of a singular function, such as $\rho_{n}^{\gamma}$, where $\gamma \to \infty$. We prove that under suitable assumptions on the initial data, solutions to the Aw-Rascle system converge towards the so-called duality solutions, which have previously found applications in other systems which exhibit compressive dynamics. We also prove that one can obtain weak solutions to the limiting system under stricter assumptions on the initial data. 
Finally, we discuss (non-)uniqueness issues.
\end{abstract}

	

\section{Introduction}
We study the following free-congested system on the spatial domain $\Omega= \R$:
\begin{subnumcases}{\label{eq:limit-system}}
\partial_t \rho + \partial_x(\rho u) = 0, \\
\partial_t(\rho u + \partial_x \pi) + \partial_x \big((\rho u + \partial_x \pi) u \big) = 0,\label{eq:limit-system-mom}\\
0 \leq \rho \leq 1, ~ (1-\rho) \pi = 0, \pi \geq 0. \label{eq:limit-system-constraint}
\end{subnumcases}
 This system is also known as the ‘hard-congestion model'. The variable $\rho$ represents the density while $u$ denotes the velocity. The  quantity $\pi$ is an unknown potential which plays a key role in the dynamics of the system. In particular, it appears in the exclusion constraint $(1-\rho)\pi = 0$ of \eqref{eq:limit-system-constraint} which reflects the two-phase nature of the system. This constraint tells us that in regions where $\rho < 1$ (the free/compressible phase) we must have $\pi = 0$ and when $\rho = 1$ (the congested/incompressible phase) the potential $\pi$ activates. The free-congested system finds many applications in the study of phenomena which involve congestion, such as traffic flow, crowd dynamics \cite{vauchelet2017incompressible} and granular flows \cite{perrin2018one}. The system \eqref{eq:limit-system} can also be related to the constrained Euler equations which have been studied by Berthelin and Preux and Maury, and has also previously been formally derived by Lefebvre-Lepot and Maury in \cite{lefebvre2008micro}. Lagrangian solutions to this system were also constructed by Perrin and Westdickenberg in \cite{perrin2018one}. For a more complete overview of recent results concerning free-congested fluid models, we refer to \cite{perrin2018overview}. Weak solutions (in the sense of distributions) to the hard-congestion model were also very recently obtained in \cite{HCLMehmood, HCL} on the domain $\Omega = \mathbb{T}$ by studying the asymptotic limit passage of solutions to the generalised Aw-Rascle system
\begin{subequations} 
\newcommand{\mystrut}{\vphantom{\pder{}{}}}
\begin{numcases}{}
    \partial_{t} \rho + \partial_{x}(\rho u) = 0, &\text{ in } $ (0, T)\times \Omega, $ \label{AR-1} \\[1ex]
    \partial_{t} (\rho w) + \partial_{x} (\rho w u) = 0, &\text{ in } $ (0, T)\times \Omega, $ \label{AR-2} \\[1ex]
    w = u + \partial_{x}p(\rho). \label{AR-3}
\end{numcases}
\end{subequations} 
The Aw-Rascle system, whose origin can be traced back to the papers by Aw and Rascle \cite{aw2000resurrection} and Zhang  \cite{zhang_non-equilibrium_2002}, models the evolution of traffic flow in one spatial dimension. Here, the quantities $u$ and $w$ respectively refer to the actual and desired velocities of agents. The global existence of solutions to the above generalised Aw-Rascle system \eqref{AR-1}-\eqref{AR-3} for $n$ fixed with the offset function
\begin{equation} \label{pnINTRO}
    p(\rho) = \rho^{\gamma_{n}}, \quad \gamma_{n} \to \infty, 
\end{equation}
was obtained in \cite{HCLMehmood}. 
There, it was shown that as $n \to \infty$, there exists a subsequence of such solutions which converges towards $(\rho, u, \pi)$ a solution of the hard-congestion model \eqref{eq:limit-system}. The offset function \eqref{pnINTRO} represents an approximation of $\pi$ for fixed $n$. Since \eqref{pnINTRO} is non-singular for $n$ fixed, the maximal density constraint is relaxed at the approximate level which allows inter-penetrability between the free and congested phases. The same limit passage was also studied in \cite{HCL} where the offset function was instead taken to be 
\begin{equation} \label{pnSINGULARintro}
    p(\rho) = \eps \frac{\rho^{\gamma}}{(1-\rho)^{\beta}}, \quad \gamma \ge 0, ~ \beta > 1, ~ \eps \to 0.
\end{equation}
In this case, the limit potential $\pi$ is approximated by the singular potential \eqref{pnSINGULARintro} which diverges as the approximate density $\rho_{n}$ approaches 1 from below. The potential plays the role of a barrier, which is physically significant. For instance, in the context of collective motion, this potential may represent the social repulsion forces exerted between agents.

 In this paper, we are interested in weaker/measure-valued solutions.
Indeed, measure-valued solutions have been shown to be relevant from the modeling point of view when studying heterogeneity properties, for instance in multi-component flows~\cite{bresch2015} or for clusters formation in traffic flows~\cite{leveque2001}.
The main goal of this paper is to extend the previous existence results for the system \eqref{eq:limit-system} to solutions where the velocity $u$ is bounded and $\partial_{x}\pi$ is a measure. 
In this situation, the equation \eqref{eq:limit-system-mom} cannot be understood in the distributional sense since the product $u \partial_{x} \pi$ is not well-defined. 
Thus, we are required to work with a new notion of solution in order to obtain meaningful existence results. 
This scenario corresponds to the case where concentration phenomena appears in the congestion component of the system \eqref{eq:limit-system}. 
Our ambition is motivated by previous works on the pressureless gas equations~\cite{franccois1999duality, boudin2000solution}. 
It is known that concentration phenomena in the density (which can be thought of as the creation of Dirac masses in finite time) is linked to the compressive nature of the dynamics, which is understood here as a control on the positive part of $\partial_{x}u$. 
The desire for a robust notion of solution in this case led to the concept of 'duality solution' which was coined by Bouchut and James~\cite{bouchut1998one},
and this notion has been successfully applied to the pressureless gas equations by Bouchut and Brenier in~\cite{franccois1999duality} and by Boudin in~\cite{boudin2000solution} through a vanishing viscosity limit.
Vauchelet and James also studied the existence of duality solutions for one-dimensional aggregation equations \cite{jamesaggregationvauch} and chemotaxis \cite{james2013chemotaxis}. 
To the best of our knowledge there are currently no results concerning duality solutions for the hard-congestion model. 
We detail the definition and main properties of duality solutions in the next section. 
Our main result is the following.
\begin{theorem}{\label{thm:main}}
       Fix $T>0$ and let $(\rho^0,m^0 ,\pi^0)$ be such that
       \begin{align}
           0 ~ \leq ~ \rho^0 \leq 1, \quad  m^0 \in \mathcal{M}_{loc}(\mathbb{R}),
           \quad  \pi^0 \in BV_{loc}(\mathbb{R}).
       \end{align}
       Then there exists a duality solution $(\rho, m = \rho u + \partial_x \pi, \pi)$ to the hard-congestion model
       \begin{subequations} \label{HCL-LIMIT}
\newcommand{\mystrut}{\vphantom{\pder{}{}}}
\begin{numcases}{}
    \partial_{t} \rho + \partial_{x} (\rho u) = 0, \qquad (0,T) \times \mathbb{R}, \label{HCL-L1}   \\[1ex]
    \partial_{t} (\rho  u + \partial_{x}\pi) + \partial_{x}((\rho u + \partial_{x}\pi)u)  = 0,  \qquad (0,T) \times \mathbb{R},\label{HCL-L2} \\[1ex]
    0 \le \rho \le 1,~ (1-\rho)\pi = 0,~ \pi \ge 0, \quad \text{ a.e. in } (0,T) \times \mathbb{R}, \label{HCL-L3}
\end{numcases}
\end{subequations}
        in the sense of Definition \ref{defndualityHCL} below.
    \end{theorem}
 Our proof is based on the same power law approximation scheme as in~\cite{HCLMehmood}. This means that we study the asymptotic limit of solutions to \eqref{AR-1}-\eqref{AR-3} with the offset function \eqref{pnINTRO}. Note that the system is set on $\R$ and not the torus. A key feature of the approximate system \eqref{AR-1}-\eqref{AR-3} with \eqref{pnINTRO} is that for fixed $n$, the system \eqref{AR-1}-\eqref{AR-2} can be formally rewritten as the one-dimensional compressible pressureless Navier-Stokes equations
\begin{subequations} \label{A}
\newcommand{\mystrut}{\vphantom{\pder{}{}}}
\begin{numcases}{}
    \partial_{t} \rho_{n} + \partial_{x} (\rho_{n} u_{n}) = 0, &\text{ in } $ (0, T)\times \mathbb{R} , $  \label{a1} \\[1ex]
    \partial_{t} (\rho_{n} u_{n}) + \partial_{x}(\rho_{n} u_{n}^{2}) - \partial_{x}(\lambda_{n}(\rho_{n})\partial_{x}u_{n}) = 0, &\text{ in } $ (0, T)\times \mathbb{R},$  \label{a2}
\end{numcases}
\end{subequations}where $\lambda_{n}(\rho_{n}) = \rho_{n}^{2}p_{n}'(\rho_{n})$. The systems \eqref{AR-1}-\eqref{AR-2} and \eqref{a1}-\eqref{a2} are equivalent for sufficiently regular solutions, and in particular for the class of regular solutions which we will consider. 

Observe that for $\rho_n < 1$ (i.e. where there is no congestion), $\lambda_{n}(\rho_{n}) \to 0$ as $n \to +\infty$, so that our study is naturally connected to the one of Boudin~\cite{boudin2000solution} on the vanishing viscosity limit towards the pressureless gas system.

\medskip

Let us now highlight the key parts of the proof of Theorem \ref{thm:main}. 
    \begin{itemize}
        \item \textbf{Global existence of approximate solutions (i.e. for fixed $n$):} this is obtained using the construction of Burtea \& Haspot \cite{Burtea_2020}. We show it is possible to obtain a lower bound on the density and this leads to global existence.
        \item \textbf{Derivation of uniform estimates:} defining the approximate potential $\pi_n = \pi_n(\rho_n)$ such that 
        \begin{equation}\label{df:pi_n}
        \pi'_n(\rho) = \rho p'_n(\rho),
        \end{equation}
    
        we obtain the control of $\pi_n(\rho_n)$ in $W^{1,1}_{loc}$ as well as an upper bound on $\rho_n$. Another crucial estimate is the one-sided Lipschitz condition on $u_n$ which leads to compressive dynamics. This is crucial to have any hope of obtaining duality solutions in the limit.
        \item \textbf{Passage to the limit:} we make use of the stability property satisfied by duality solutions as well as compensated compactness arguments to demonstrate that our approximate solutions converge in some sense towards duality solutions to \eqref{HCL-LIMIT}.  
    \end{itemize}

Finally, let us also mention that we are able to prove the existence of weak solutions to the limit system if we impose additional conditions on the initial data (see Theorem \ref{limitexistence}). This extends the results of \cite{HCL, HCLMehmood} which were on $\mathbb{T}$, the one-dimensional torus. 

The paper is organised as follows. 
In Section~\ref{sec:duality-df}, we give an overview of the theory of duality solutions and provide a definition of duality solutions for our system~\eqref{eq:limit-system}. 
In Section~\ref{sec:mainthm} we state more precisely our main results, while the global existence of regular solutions for fixed $n$ is left to Section~\ref{sec:existence-n}. 
The uniform estimates needed to complete the limit passage are derived in Section~\ref{sectionUNIF}. 
In Section~\ref{sec:limit}, we pass to the limit and prove the existence of both weak and duality solutions to the hard-congestion model. 
Lastly, we discuss the relationship between weak and duality solutions in addition to the matter of uniqueness in Section~\ref{sec:final}.

\subsection{Duality solutions}{\label{sec:duality-df}}
The theory of duality solutions was introduced by Bouchut and James \cite{bouchut1998one} to provide a notion of solution for the transport equation
\begin{equation} \label{ncintro}
    \begin{cases}
 \partial_{t}u+a(t,x)\partial_{x}u=0 & \text{ in }(0,T) \times \mathbb{R}, \\
u(x,0) = u^{0}(x) \in BV_{loc}(\mathbb{R}),
\end{cases}
\end{equation} and its conservative counterpart
\begin{equation} \label{cintro}
    \begin{cases}
 \partial_{t}\mu+\partial_{x}(a(t,x)\mu)=0 & \text{ in }(0,T) \times \mathbb{R}, \\
\mu(x,0) = \mu^{0}(x) \in \mathcal{M}_{loc}(\mathbb{R}),
\end{cases}
\end{equation}
where $a \in L^{\infty}((0,T) \times \mathbb{R})$ satisfies the one-sided Lipschitz condition (OSLC)
\begin{equation} \label{OSLC}
\partial_{x}u \le \alpha \qquad \text{ in } \mathcal{D}',
\end{equation}
where $\alpha \in L^{1}(0,T)$. The issue with problem \eqref{ncintro} (and therefore with \eqref{cintro}) is that the product $a\partial_{x}u$ is not defined in general, since $a$ is an $L^{\infty}$ function and $\mu := \partial_{x}u$ a finite Radon measure. Thus, we cannot understand equations \eqref{ncintro}-\eqref{cintro} in the distributional sense and so we must look for a new notion of solution - the duality solution. Let us concentrate our attention on the conservative problem \eqref{cintro}. 
The definition of a duality solution may be formally motivated as follows. Suppose we are looking for a weak solution to problem \eqref{cintro}. Then multiplying by a test function $p$ which is sufficiently smooth and integrating by parts, we get for each $t \in [0,T]$,
\begin{align} \label{WFdual}
   - \int_{0}^{t}\int_{\mathbb{R}} \mu(s,x) (\partial_{t}p + a(s,x)\partial_{x}p)~dxds + \int_{\mathbb{R}}p(t,x)\mu(t,x)dx -\int_{\mathbb{R}}p(0,x)\mu(0,x)dx = 0.
\end{align}
A natural way to formulate a notion of weak (duality) solution is therefore as follows. We look for a measure $\mu$ which satisfies \begin{equation} \label{dualitydefnNAIVE}
    \frac{d}{dt}\int_{\mathbb{R}}p(t,x)\mu(t,x)dx = 0
\end{equation} for all $p$ which solve the backward transport equation
\begin{equation} \label{bpr}
    \partial_{t}p+a\partial_{x}p=0, ~ ~ ~p(T,\cdot)=p^{T}.
\end{equation} Note that \eqref{bpr} is the formal adjoint (or 'dual') of \eqref{cintro}. In order for our notion of duality solution to possess useful properties (e.g. stability) we must choose the function space for which we look for solutions to problem \eqref{bpr} carefully. We assume $p^{T} \in Lip_{loc}(\mathbb{R})$ and look for solutions $p \in Lip_{loc}((0,T) \times \mathbb{R})$. The space of such solutions is denoted by $\mathcal{L}$.  It is well-known that the OSLC \eqref{OSLC} guarantees the existence of solutions to \eqref{bpr}. The issue that immediately arises, however, is that there is no uniqueness in the class of solutions $p \in \mathcal{L}$, as highlighted by the counterexample $a(t,x) = -\text{sgn}(x)$ considered in \cite{bouchut1998one, lions2023transport}. Uniqueness in the space of test functions is pivotal if one is to obtain weak stability results (e.g. Theorem 4.3.2. of \cite{bouchut1998one}). Arguing by stability is the canonical way to prove the existence of duality solutions in practice since it allows us to more conveniently verify the definition \eqref{dualitydefnNAIVE}. The highlight of \cite{bouchut1998one} is therefore the introduction of the class of 'reversible solutions', which is a sub-class of $\mathcal{L}$ in which we have uniqueness. 
\begin{definition}[Reversible solutions]
    A solution $p \in \mathcal{L}$ is \textit{reversible} if $p$ is locally constant on the open set \begin{equation}
        \mathcal{V}_{e} = \{ (t,x) \in (0,T) \times \mathbb{R} : \exists ~p_{e} \in \mathcal{L} \text{ with } p_{e}(T,\cdot)=0 \text{ and } p_{e}(t,x) \ne 0               \}.
    \end{equation}
\end{definition}
The class of reversible solutions also gives way to a generalised backward flow.
\begin{definition}[Generalised backward flow \cite{bouchut1998one}] 
Let $T> 0$, $s \in (0,T]$ and $D_{b} = \{(t  \in \mathbb{R} ; 0 \le t \le s \}$.
The backward flow $X(s,\cdot,\cdot) \in Lip(D_{b} \times \mathbb{R})$ is defined as the unique reversible solution to 
\begin{subequations} 
\newcommand{\mystrut}{\vphantom{\pder{}{}}}
\begin{numcases}{} 
    \partial_{t} X  + a(t,x)\partial_{x}X = 0,~~  \text{ in } (0,s) \times \mathbb{R}, \\[1ex] 
   X(s,s,x) = x. 
\end{numcases}
\end{subequations} If $s=0$ then we set $X(0,0,x)=x$. \end{definition}
One can find more convenient characterisations of reversible solutions in \cite{bouchut1998one}. We can now define duality solutions. For this purpose we introduce the spaces
\begin{align*}
    &\mathcal{S}_{\mathcal{M}} := C([0,T]; \mathcal{M}_{loc,w}(\mathbb{R})),  \\[1ex]
    &\mathcal{S}_{L} := C([0,T]; L^{\infty}_{w^{\star}}(\mathbb{R})), \\[1ex]
    &\mathcal{T}_{BV} :=  L^{\infty}(0,T; BV_{loc}(\mathbb{R})).
\end{align*} Here, the subscripts $X_{w}$ and $X_{w^{\star}}$ are used to denote the function space $X$ with the weak and weak-$\star$ topologies respectively.
\begin{definition}[Duality solutions] \label{defnDualityORIGINAL}
    We say that $\mu \in \mathcal{S}_{\mathcal{M}}$ is a duality solution to \eqref{cintro} if for any $\tau \in (0,T]$ and any reversible solution $p$ with compact support in space to $\partial_{t}p + a\partial_{x}p = 0$ in $(0,\tau) \times \mathbb{R}$, the function $t \mapsto \int_{\mathbb{R}} p(t,x)\mu(t,dx)$ is constant on $[0,\tau]$.
\end{definition}
We refer the reader to \cite{bouchut1998one} for a more complete account of the theory of duality solutions. 

\begin{remark}
    The generalised backward flow leads to another characterisation for duality solutions. We may say that $\mu \in \mathcal{S}_{M}$ is a duality solution to $\partial_{t}\mu + \partial_{x}(a\mu) = 0$ with initial data $\mu^{0}$ if for all $t \in [0,T]$ and $g \in C_{c}(\mathbb{R})$,
    \begin{equation}
        \int_{\mathbb{R}}g(x)\mu(t,dx) = \int_{\mathbb{R}} g(X(0,t,x)) \mu^{0}(dx).
    \end{equation} This is taken as the definition of duality solutions by Lions and Seeger in \cite{lions2023transport}. Note that the flow generated by the reversible solutions coincides with the Filippov flow introduced in \cite{filippov1960differential}. It is also worthwhile to mention that in \cite{poupaud1997measure} Poupaud and Rascle used properties of the Filippov flow to prove the existence of a unique measure-valued solution to the conservation law \eqref{cintro}. This solution in fact coincides with the notion of duality solution.
\end{remark}
The following theorem collects the relevant results concerning duality solutions that we will need.
\begin{theorem}[Bouchut and James \cite{bouchut1998one}] \hphantom \\ \label{DUALITYresults}
    \begin{enumerate}
        \item Given $\mu^{0} \in \mathcal{M}_{loc}(\mathbb{R})$, there exists a unique $\mu \in \mathcal{S}_{\mathcal{M}}$ duality solution to \eqref{cintro} such that $\mu(0,\cdot) = \mu^{0}$.
        \item  The flux associated to a duality solution $\mu$ is denoted $a ~ \triangle ~\mu$ and is defined as \begin{equation}
            a  ~\triangle ~\mu := - \partial_{t} u,
        \end{equation} where $u(t,x) := - \int_{-\infty}^{x} \mu(t,y)~dy$. There exists a bounded Borel function $\hat{a}$ (known as the 'universal representative of $a$') such that $a = \hat{a}$ a.e. and $a ~ \triangle ~\mu = \hat{a}\mu$. This allows us to recover the distributional equality
        \begin{equation}
            \partial_{t}\mu + \partial_{x}(\hat{a}\mu) = 0, ~~ \text{ in } \mathcal{D}'.
        \end{equation}
        \item   Let $\{a_{n}\} \subset L^{\infty}((0,T) \times \mathbb{R})$ be a bounded sequence with $a_{n} \rightharpoonup^{*} a$ weak-* in $L^{\infty}((0,T) \times \mathbb{R})$. Assume $\partial_{x}a_{n} \le \alpha_{n}(t)$ where $\{\alpha_{n}\}$ is bounded in $L^{1}(0,T)$ with $\partial_{x}a \le \alpha \in L^{1}(0,T)$. Consider a sequence $\{\mu_{n}\}$ of duality solutions to \begin{equation}
        \partial_{t}\mu_{n} + \partial_{x}(a_{n}\mu_{n})= 0 \text{ in } (0,T) \times \mathbb{R},
    \end{equation}where $\mu_{n}(0, \cdot)$ is bounded in $\mathcal{M}_{loc}(\mathbb{R})$ and $\mu_{n}(0, \cdot) \rightharpoonup \mu^{0} \in \mathcal{M}_{loc}(\mathbb{R})$. Then $\mu_{n} \to \mu$ in $\mathcal{S}_{\mathcal{M}}$, where $\mu$ is the duality solution to \begin{equation}
        \partial_{t}\mu + \partial_{x}(a\mu)= 0 \text{ in } (0,T) \times \mathbb{R}, ~~ \mu(0,\cdot) = \mu^{0}.
    \end{equation}
    Moreover, $\hat{a}_{n}\mu_{n} \rightharpoonup \hat{a}\mu$ weakly in $\mathcal{M}_{loc}((0,T) \times \mathbb{R}).$
    \end{enumerate}
\end{theorem}
Finally, we propose the following definition of duality solutions for the hard-congestion model \eqref{HCL-L1}-\eqref{HCL-L3}.
\begin{definition} \label{defndualityHCL}
We say that $(\rho, m,\pi) \in \mathcal{S}_{L} \times \mathcal{S}_{\mathcal{M}} \times \mathcal{T}_{BV}$ is a duality solution to \eqref{HCL-L1}-\eqref{HCL-L3} on $[0,T]$ with initial data $(\rho^{0}, m^{0}, \pi^{0}) \in L^{\infty}(\mathbb{R}) \times \mathcal{M}_{loc}(\mathbb{R}) \times BV_{loc}(\mathbb{R})$ if there exists $u \in  L^{\infty}((0,T) \times \mathbb{R})$ and $\alpha \in L^{1}(0,T)$ with $\partial_{x}u \le \alpha $ on $[0,T] \times \mathbb{R}$ such that
    \begin{enumerate}[label=(\roman*)]
        \item $\rho$ is a distributional solution to $\partial_{t}\rho + \partial_{x}(\rho u ) = 0$ with $\rho(0, \cdot) = \rho^{0}$ and $0 \le \rho \le 1$,
        \item $m$ is a duality solution to $\partial_{t}m + \partial_{x}(m u ) = 0$ with $m(0,\cdot) = m^{0}$, in the sense of Definition \ref{defnDualityORIGINAL}, 
        \item the switching relation $(1-\rho)\pi = 0$ holds a.e. in $[0,T] \times \mathbb{R}$ with $\pi \ge 0$,
         \item the equality $m = \rho u + \partial_{x}\pi$ holds in the sense of measures.
    \end{enumerate}
\end{definition}
Notice that this differs from Definition \eqref{defnDualityORIGINAL} since we now have a system of equations and the velocity $u$ is an unknown obtained in the limit. Conditions (i) and (ii) are natural, while (iii) corresponds to \eqref{HCL-L3} and (iv) represents the recovery of the non-linear coupling between the equations. This definition is reminiscent of the definition provided by Bouchut and James in \cite{franccois1999duality} for a system of pressureless gases. 

\subsection{Main results}{\label{sec:mainthm}}
In this paper we adopt the Bochner space notation $X_{t}Y_{x} := X(0,T; Y(\mathbb{R}))$ for appropriate function spaces $X$ and $Y$. We first provide a precise definition of regular solutions to the system 
\eqref{a1}-\eqref{a2}, which are also classical.
\begin{definition}[Regular solutions] \label{defregsoln}
    Suppose $n \in \mathbb{N}$ is fixed and $p_{n}$ is given by \eqref{pnINTRO} and $\overline{\rho} \in (0,1]$. Assume further that $(\rho_{n}^{0}, u_{n}^{0}) \in H^{3}(\mathbb{R}) \times H^{3}(\mathbb{R})$ and $0 < \rho_{n}^{0}$. The pair $(\rho_{n}, u_{n})$ is called a regular solution to \eqref{a1}-\eqref{a2} on $[0,T]$ if \begin{equation*}
        \rho_{n} - \overline{\rho} \in C(0,T; H^{3}(\mathbb{R})), ~u_{n} \in C(0,T; H^{3}(\mathbb{R})) \cap L^{2}(0, T ; H^{4}(\mathbb{R))},
    \end{equation*} and $(\rho_{n}, u_{n})$ satisfy \eqref{a1}-\eqref{a2} in $\mathbb{R} \times [0,T]$. The pair $(\rho_{n}, u_{n})$ is known as a global regular solution to \eqref{a1}-\eqref{a2} if it is a regular solution on $[0,T]$ for any $T>0$.
\end{definition}
\begin{remark}
    Our definition of regular solutions is at the $H^{3}$ level so that the computations we carry out when obtaining uniform estimates are easily justified. Through a more refined argument one could work with less regular solutions and still obtain the same results. However, we do not go to such lengths in this paper since our focus is concentrated on solutions to the limit system.
\end{remark}
The first result concerns the global existence of regular solutions for $n$ fixed.
\begin{theorem}[Global existence for $n$ fixed]  \label{globalexistence}
      Let $\overline{\rho} \in (0,1]$ and $\rho_{n}^{0}, u_{n}^{0} : \mathbb{R} \to \mathbb{R}$ be such that $\rho_{n}^{0} > 0$ and \begin{equation}
        \left(\frac{1}{\rho_{n}^{0}} - (\overline{\rho})^{-1}, \rho_{n}^{0} - \overline{\rho}, u_{n}^{0}\right) \in (H^{3}(\mathbb{R}))^{3}. \label{initialreg}
    \end{equation} Then there exists a pair $(\rho_{n}, u_{n})$  with initial data $(\rho_n^{0}, u_n^{0})$ which is a global regular solution to \eqref{a1}-\eqref{a2} in the sense of Definition \ref{defregsoln}.
\end{theorem}
%
%
\begin{remark}
    This is a special case of the result proved by Burtea and Haspot \cite{Burtea_2020}. For fixed $n$ our system is very similar to theirs, with the only difference being that we have zero pressure. Nonetheless, we can follow the same steps to obtain global existence and even uniqueness, although we are not concerned with the latter for $n$ fixed. Notice that our initial density $\rho_{n}^{0}$ has a non-zero background state. In other words, $\rho_{n}^{0}(x) \to \overline{\rho} > 0$ as $|x| \to \infty$. This is chosen because we want to obtain strong solutions for fixed $n$ in order to justify later computations (i.e. when we prove the next two theorems). Choosing a zero background state is problematic since to the best of our knowledge it is not currently known whether strong solutions exist to the compressible Navier-Stokes system on the real line when the initial density has a zero background state.
\end{remark}


\begin{theorem}[Existence of a duality solution to the hard-congestion model] \label{existencedual}
Let $T>0$ and $(\rho^0, m^0, \pi^{0}) \in \mathcal{S}_{L} \times \mathcal{S}_{M} \times \mathcal{T}_{BV}$. Suppose there exists a sequence of smooth initial data $(\rho_{n}^{0}, w_{n}^{0})$ with \begin{align}
        &\rho_{n}^{0} \rightharpoonup \rho^{0} \text{ weakly-* in } L^{\infty}_{loc}(\mathbb{R}), \label{Didconvergence1} \\[1ex]
        &\rho_{n}^{0}w_{n}^{0} \rightharpoonup m^{0} \text{ weakly in } \mathcal{M}_{loc}(\mathbb{R}), \label{Didconvergence2} \\[1ex]
    &\partial_{x}\pi_{n}^{0} \rightharpoonup \partial_{x}\pi^{0} \text{ weakly in } \mathcal{M}_{loc}(\mathbb{R})\label{Didconvergence3}.
    \end{align}
Furthermore, assume that $(\rho_n^{0}, u_n^{0} = w_{n}^{0} - \partial_{x}p_{n}(\rho_{n}^{0}))$ satisfies \eqref{initialreg} and that there exists $C>0$ independent of $n$ and a sequence $\{r_{n}^{0}\}_{n} \subset (0,1]$ with $r_{n}^{0} \to r^{0} \in [0,1]$ such that
\begin{align}
 &0 < r_{n}^{0} \le \rho_{n}^{0}(x) \le (1 + C)^{\frac{1}{\gamma_{n}+1}}  ~~\forall ~x \in \mathbb{R}, \label{A1rho} \\[1ex]
&  \|\rho_{n}^{0}u_{n}^{0}\|_{L^{1}_{x}} + \|\rho_{n}^{0}w_{n}^{0}\|_{L^{1}_{x}}  \le C, \label{A2mom} \\[1ex]
&\esssup_{x\in \mathbb{R}} (\lambda_{n}(\rho_{n}^{0})\partial_{x}u_{n}^{0}) \le \frac{\gamma_{n}}{\big(M_{n}^{0}T + (r_{n}^{0})^{-1}\big)^{\gamma_n+1} } , \label{A3A} \\[1ex]
& \|u_{n}^{0}\|_{L^{\infty}_{x}} \le C, \label{A4uLinf}
    \end{align} where \begin{equation}
        M_{n}^{0} := \esssup_{x \in \mathbb{R}} \frac{\partial_{x}w_{n}^{0}}{\rho_{n}^{0}}.
    \end{equation}
    Then, the solution $(\rho_{n}, u_{n})$ on $[0,T]$ established in Theorem \ref{globalexistence} satisfies the following uniform bounds:
\begin{align}
    &\|\rho_{n}u_{n}\|_{L^{\infty}_{t}L^{1}_{x}} + \|\rho_{n}w_{n}\|_{L^{\infty}_{t}L^{1}_{x}} \le C, \label{unif1MOM}\\[1ex]
    &\partial_{x}u_{n} \le C,  \label{unif2OSL}
\end{align}
where $C>0$ is a generic constant independent of $n$, and for any compact $K \subset \R$ there exists $C' = C'(K) > 0$ independent of $n$ such that
\begin{align}
    &\|\rho_{n}\|_{L^{\infty}((0,T) \times K)} + \|u_{n}\|_{L^{\infty}(0,T; W^{1,1}(K))} +\|\pi_{n}\|_{L^{\infty}(0,T; W^{1,1}(K))} \le C'. \label{unif3RHOU} 
\end{align}
    Moreover, there exists a subsequence $(\rho_{n}, u_{n}, \pi_{n})$ of solutions to \eqref{AR-1}-\eqref{AR-2} with initial data $(\rho_{n}^{0}, u_{n}^{0}, \pi_{n}^{0})$ which converges to $(\rho , m, \pi)$, a duality solution of \eqref{HCL-L1}-\eqref{HCL-L3} with initial data $(\rho^{0}, m^{0}, \pi^{0})$. 
    Eventually, the entropy inequality
    \begin{equation} \label{dualityentropy}
        \partial_{t}|\mu| + \partial_{x}(\hat{u}|\mu|) \le 0, ~~ \text{ in } \mathcal{D}',
    \end{equation}
    is satisfied by $\mu = \rho, m$, where $\hat u$ is a universal representative of $u$ (see Theorem~\ref{DUALITYresults}).
\end{theorem}

\bigskip


\begin{remark} \label{remarkDUALITY}
    As said in the introduction, the need to introduce the notion of a duality solution for the limiting system arises from the observation that the product $u \partial_{x} \pi$ which appears in \eqref{HCL-L2} is not well-defined if $u$ is a discontinuous function and $\partial_{x} \pi$ is a measure. Indeed, under the bounds \eqref{unif1MOM}-\eqref{unif3RHOU} it may happen that (up to a subsequence) $u_{n}$ converges to a discontinuous function and $\partial_{x}\pi_{n}$ converges to a measure. In such a case we cannot make sense of the hard-congestion model in the distributional sense and so we turn to the theory of duality solutions.
\end{remark}
If we additionally assume the uniform control of the norm $\|\sqrt{\rho_{n}^{0}}w_{n}^{0}\|_{L^{2}_{x}}$, we can assert the existence of weak solutions to the hard-congestion model.

    \begin{theorem}[Existence of a weak solution to the hard-congestion model] \label{limitexistence} 
    Assume $T>0$ and that $(\rho_n^{0}, u_n^{0})$ satisfies \eqref{initialreg} Furthermore, assume that there exists $C>0$ independent of $n$ 
    such that \eqref{A1rho}-\eqref{A4uLinf} hold as well as the additional condition
    \begin{equation}
        \|\sqrt{\rho_{n}^{0}}w_{n}^{0}\|_{L^{2}_{x}} \le C. \label{A5w}
    \end{equation} Then the solution $(\rho_{n}, u_{n})$ obtained in Theorem \ref{globalexistence} satisfies the uniform bounds \eqref{unif1MOM}-\eqref{unif3RHOU} and additionally
    \begin{align}
         &\|\pi_{n}\|_{L^{\infty}(0,T; W^{1,2}_{loc}(\mathbb{R}))} \le C, \label{unif4PI_L2} \\[1ex]
          &\|\sqrt{\rho_{n}}w_{n}\|_{L^{\infty}_{t}L^{2}_{x}} \le C, \label{unif5W}
    \end{align}
    where $C>0$ is independent of $n$. Suppose additionally that
    \begin{align}
        &\rho_{n}^{0} \rightharpoonup \rho^{0} \text{ weakly in } L^{2}(\mathbb{R}), \label{idconvergence1} \\[1ex]
        &\rho_{n}^{0}u_{n}^{0} \rightharpoonup \rho^{0}u^{0} \text{ weakly in } L^{2}(\mathbb{R}), \label{idconvergence2} \\[1ex]
        &\partial_{x} \pi_{n}(\rho_{n}^{0}) \rightharpoonup \partial_{x} \pi^0  \text{ weakly in } L^{2}(\mathbb{R}) \label{idconvergence3}.
    \end{align} Then there exists a subsequence $(\rho_{n}, u_{n})$ of solutions to \eqref{a1}-\eqref{a2} with initial data $(\rho_{n}^{0}, u_{n}^{0}, \pi_{n}^{0})$ which converges to $(\rho, u, \pi)$, a weak solution of \eqref{HCL-L1}-\eqref{HCL-L3} with initial data $(\rho^{0}, u^{0}, \pi^{0})$. The following entropy conditions hold for the limiting system:
    \begin{itemize}
        \item one-sided Lipschitz condition: \begin{equation} \label{OSLlimit}
            \partial_{x} u \le C ~~ \text{ in } \mathcal{D}',
        \end{equation}
        \item entropy inequality: \begin{equation}
            \partial_{t}(\rho S(u)) + \partial_{x}(\rho u S(u)) - \partial_{x} \Lambda_{S} \le 0 ~~ \text{ in } \mathcal{D}',
        \end{equation} for any convex $S \in C^{1}(\mathbb{R})$ and $\Lambda_{S} \in \mathcal{M}((0,T) \times \mathbb{R})$ with $|\Lambda_{S}| \le Lip_{S}|\Lambda|$, where $-\Lambda = -\overline{\lambda(\rho)\partial_{x}u} \in \mathcal{M}^{+}((0,T) \times \mathbb{R})$.
    \end{itemize}
\end{theorem}
\begin{remark}
    The above result is the $\mathbb{R}-$analogue to Theorem 2.2, \cite{HCL} and Theorem 1.4, \cite{HCLMehmood}, which were both on the one-dimensional torus. The uniform assumptions \eqref{A1rho}-\eqref{A5w} are slightly stronger than those in \cite{HCL} since we can no longer exploit the boundedness of the domain. Indeed, the majority of the uniform estimates in the aforementioned works do not carry over to an unbounded domain. Nonetheless we show that using maximum-principle and regularisation arguments, we can still acquire the bounds needed to complete the limit passage on the real line.
\end{remark}
\begin{remark}
    The assumption \eqref{A5w} is crucial if we wish to obtain a weak solution to the limiting system. It implies the uniform $L^{2}_{x}$ control on $\partial_{x}\pi_{n}(\rho_{n}^{0})$ which prevents a measure from being generated in the initial layer (i.e. at $t=0$). This $L^{2}_{x}$ control can then be propagated for positive times which means that the weak limit of $\partial_{x}\pi_{n}$ is an $L^{2}_{x}$ function unlike in the previous theorem. Thus, the product $u \partial_{x} \pi$ is now well-defined and our limit $(\rho, u, \pi)$ is in fact a weak (distributional) solution to the limiting system.
\end{remark}
\begin{remark}
    The assumption \eqref{A1rho} is needed to control the potential $\pi_{n}$ in \eqref{unif3RHOU}, while \eqref{A2mom} gives us \eqref{unif1MOM}. The bound on the singular viscosity term \eqref{A3A} leads to the OSL condition \eqref{unif2OSL}. A very similar bound to \eqref{A3A} was also seen in \cite{HCL}. The bounds \eqref{A4uLinf} and \eqref{A5w} are propagated for positive times, leading to \eqref{unif5W} and the latter bound in \eqref{unif3RHOU}.
\end{remark}

\subsection{Outline of the paper}
In Section~\ref{sec:existence-n} we focus on obtaining a global-in-time regular solution for $n$ fixed. 
The overall strategy is similar to what can be found in \cite{Burtea_2020,HCLMehmood}. 
In Section~\ref{sectionUNIF} we carry out uniform in $n$ estimates which will be needed to prove both Theorems \ref{limitexistence} and \ref{existencedual}. 
Then we complete the proofs of Theorems \ref{limitexistence} and \ref{existencedual} in Section~\ref{sec:limit}.

\section{Global existence of regular solutions for \texorpdfstring{$n$}{Lg} fixed}{\label{sec:existence-n}}
This section is split up into four parts. We first establish the existence of a local-in-time solution in Subsection~\ref{sec:local-ex} before carrying out standard energy estimates in Subsection~\ref{sec:energy}. 
In Section~\ref{sec:blowup} we prove a blow-up result which is the real-line analogue to Theorem 1.1. of \cite{Constantin_2020} or Lemma 2.6. of \cite{HCLMehmood}. 
In Section~\ref{sec:global-ex} we complete the proof of global existence.
\subsection{Local existence of regular solutions}{\label{sec:local-ex}}
The local existence of regular solutions is given in the following theorem. This is a classical result and therefore we omit the proof.
\begin{theorem}[Local existence for $n$ fixed] \label{localexistence}
    Let $\overline{\rho} \in (0,1]$ and $\rho_{n}^{0}, u_{n}^{0} : \mathbb{R} \to \mathbb{R}$ be such that $\rho_{n}^{0} > 0$ and \eqref{initialreg} holds. Then there exists $T^{*}>0$ and a pair $(\rho_{n}, u_{n})$ which is a regular solution to \eqref{a1}-\eqref{a2} on $[0,T^{*})$ with initial data $(\rho_{n}^{0}, u_{n}^{0})$ such that $\rho_{n} \ge 0$ on $[0,T^{\star})$.
\end{theorem}
\subsection{Energy estimates}{\label{sec:energy}}
We now introduce some energy estimates which are important when extending the solution from local to global. The first estimate is obtained by multiplying \eqref{a2} by $u_{n}$ and integrating by parts in space and time.
\begin{lemma}[Basic energy]\label{basicenergy}
    Assume that $(\rho_{n}, u_{n})$ is a regular solution to \eqref{A} on the time interval $[0,T]$. Then,
    \begin{equation} \vspace{-10pt}
        \label{E2} 
        \| \sqrt{\rho_{n}}u_{n}\|^{2}_{L^{\infty}_{t}L^{2}_{x}} + 2\| \sqrt{\lambda_{n}(\rho_{n})}\partial_{x}u_{n}\|^{2}_{L^{2}_{t}L^{2}_{x}} = \underbrace{\| \sqrt{\rho_{n}^{0}}u_{n}^{0}\|^{2}_{L^{2}_{x}}}_{=: E_{1}}
    \end{equation} for all $t \in [0,T]$.
\end{lemma}
For the next estimate, we introduce the quantity
\begin{equation}
    H_{n}(\rho_{n}) := \frac{1}{\gamma_{n}+1}\rho_{n}^{\gamma_{n}+1}
\end{equation} 
and the relative functional
\begin{equation}
    H_{n}(\rho_{n} \vert \overline{\rho}) := \frac{1}{\gamma_{n}+1}\rho_{n}^{\gamma_{n}+1} - (\rho_{n} - \overline{\rho})\overline{\rho}^{\gamma_{n}} - \frac{1}{\gamma_{n}+1}\overline{\rho}^{\gamma_{n}+1},
\end{equation} 
associated to the limit value $\overline{\rho}$ prescribed at infinity.
\begin{lemma}[Relative energy] Assume that $(\rho_{n}, u_{n})$ is a regular solution to \eqref{A} on the time interval $[0,T]$. Then,
    \begin{equation}
        \sup_t\int_{\mathbb{R}} H_{n}(\rho_{n}|\overline{\rho})~dx + \|\sqrt{\rho_{n}}w_{n}\|_{L^{\infty}_{t}L^{2}_{x}}^{2} + \|\sqrt{\rho_{n}}\partial_{x}p_{n}(\rho_{n})\|_{L^{2}_{t,x}}^{2} \le \|\sqrt{\rho_{n}^{0}}w_{n}^{0}\|_{L^{2}_{x}}^{2}
        +\int_{\mathbb{R}} H_{n}(\rho_{n}^0|\overline{\rho})~dx.
    \end{equation}
\end{lemma}
\begin{proof}
First, multiplying \eqref{AR-2} by $w_{n}$ and integrating by parts gives
\begin{equation} \label{rhowENERGYest}
    \|\sqrt{\rho_{n}}w_{n}\|_{L^{\infty}_{t}L^{2}_{x}} \le \|\sqrt{\rho_{n}^{0}}w_{n}^{0}\|_{L^{2}_{x}}.
\end{equation} Next we recall that due to the relationship $w_{n} = u_{n} + \partial_{x}p_{n}(\rho_n)$ the continuity equation can be rewritten as
\begin{equation}
    \partial_{t}\rho_{n} + \partial_{x}(\rho_{n}w_{n}) = \partial_{x}(\rho_{n}\partial_{x}p_{n}(\rho_{n})).
\end{equation} Multiplying this equality by $H_{n}'(\rho_{n}) = p_{n}(\rho_{n})$ and using the chain rule, we have
\begin{equation*}
    \partial_{t}H_{n}(\rho_{n}) + p_{n}(\rho_n)\partial_{x}(\rho_{n}w_{n}) - p_{n}(\rho_n)\partial_{x}(\rho_{n}\partial_{x}p_{n})=0.
\end{equation*} 
Adding and subtracting $(\rho_{n} - \overline{\rho})H_{n}'(\overline{\rho}) + H_{n}(\overline{\rho})$ and integrating over space leads to
\begin{equation*}
    \frac{d}{dt}\int_{\mathbb{R}} H_{n}(\rho_{n} | \overline{\rho}) + \frac{d}{dt}\int_{\mathbb{R}} (\rho_{n} - \overline{\rho}) ~dx + \frac{1}{2}\|\sqrt{\rho_{n}}\partial_{x}p_{n}(\rho_{n})\|_{L^{2}_{x}} 
    \leq \frac{1}{2}\|\sqrt{\rho_{n}}w_{n}\|_{L^{2}_{x}}.
\end{equation*} 
The result follows from noticing that the second term on the left hand-side vanishes due to the conservation of mass.
\end{proof}
\subsection{The blow-up lemma}{\label{sec:blowup}}
We first introduce a blow-up criterion.
\begin{lemma}[Criteria for blow-up of regular solutions] \label{blowup}
    Suppose $(\rho_{n}, u_{n})$ is a regular solution to \eqref{a1} - \eqref{a2} on $[0,T^{*})$ with initial data $(\rho_{n}^{0}, u_{n}^{0})$ satisfying \eqref{initialreg}. Then provided that 
    \begin{equation}
        \underline{\rho_{n}} := \inf_{t \in [0,T^{*})} \inf_{x \in \mathbb{R}} \rho_{n}(t,x) > 0, 
    \end{equation}
    we have that
    \begin{equation} \label{blow-up-statement}
        \sup_{t \in [0,T^{*})} \|\rho_{n}\|_{L^{\infty}(0,t; H^{4})} + \sup_{t \in [0,T^{*})} \|u_{n}\|_{L^{\infty}(0,t; H^{4})} + \sup_{t \in [0,T^{*})} \|u_{n}\|_{L^{2}(0,t; H^{5})} < +\infty,
    \end{equation} and therefore the solution can be extended to a larger time interval $[0,T)$, where $T > T^{*}$. In other words, the solution does not lose regularity unless the density reaches $0$ somewhere in the domain.
\end{lemma}
This result has already been shown in the case of $\Omega = \mathbb{T}$ in \cite{Constantin_2020, HCLMehmood}, while the case $\Omega = \mathbb{R}$ can be found in Theorem 3.3. of \cite{Burtea_2020}. 
\subsection{Proving the existence of a global solution}{\label{sec:global-ex}}
In this subsection we complete the proof of Theorem~\ref{globalexistence}. The proof follows a very similar 'maximum-principle' strategy which was seen in the proof of Theorem 1.2. in \cite{HCLMehmood} or \cite{Constantin_2020}, where the domain was $\Omega = \mathbb{T}$, the one-dimensional torus. Nonetheless, many details of our argument are the same so we will only outline the main steps and refer the reader to \cite{HCLMehmood} for the technical details. The use of a maximum principle on the real line requires some justification, however, and therefore we will elaborate upon this point further.
\begin{proof}
   The local existence result (Theorem \ref{localexistence}) gave us the existence of a solution $(\rho_{n}, u_{n})$ on $\mathbb{R} \times [0,T_{0}$ where $T_{0} > 0$ and $\rho_{n} \ge c > 0$ on this interval. We denote the maximal time of existence by $T^{*}$, and assume for the sake of a contradiction that $T^{*} < +\infty$. Then the blow-up lemma implies that we must have \begin{equation} \label{densityposassumption}
        \rho_{n}(t, \cdot) > 0 \text{ for each } t \in [0,T^{*}),
    \end{equation} and \begin{equation} \label{buassumption}
          \inf_{t \in [0,T^{*})} \inf_{x \in \mathbb{R}} \rho_{n}(t,x) > 0.
    \end{equation} Therefore, proving that the density is bounded from below by a positive constant is sufficient to conclude that the solution we have obtained is in fact global-in-time. We first note that the evolution equation for the potential $W_{n} := \rho_{n}^{-1}\partial_{x}w_{n}$ is given by \begin{equation}
        \partial_{t}W_{n} + u_{n}\partial_{x}W_{n} = 0. \label{Wneqn}
    \end{equation} Since $W_{n}$ satisfies a transport equation it follows that \begin{equation} \label{Wn_cons}
         \esssup_{x \in \mathbb{R}} \frac{\partial_{x}w_{n}}{\rho_{n}}(t,x) = \esssup_{x \in \mathbb{R}} \frac{\partial_{x}w_{n}^{0}}{\rho_{n}^{0}}(x) =: M_{n}^{0}.
    \end{equation}
    Next, we carry out a maximum-principle argument with $1/\rho_{n}$. The evolution equation reads
    \begin{equation} \label{1rhoeqn}
    \partial_{t} \left( \frac{1}{\rho_{n}} \right) + u_{n} \partial_{x} \left( \frac{1}{\rho_{n}} \right) = \frac{1}{\rho_{n}} \partial_{x} u_{n}.
\end{equation} We define $
    P_{n}(t,x) := 1/\rho_{n}(t,x)$ and the corresponding maximum function \begin{equation} \label{PnMAX}
    P_{n}^{M}(t) := \max_{x \in \mathbb{R}}  \frac{1}{\rho_{n}(t,x)} \text{ on } [0,T^{*})\end{equation} with the aim of showing that \begin{equation}
       P_{n}^{M}(t) = \max_{x \in \mathbb{R}} \frac{1}{\rho_{n}^{0}}(x), ~~ \forall ~t \in [0,T].
    \end{equation} Thanks to assumption \eqref{densityposassumption} we have that for any $t \in [0,T^{*})$, $1/\rho_{n}$ has the same regularity as $\rho_{n}$. 
    Therefore, $(1/\rho_{n} - 1/\overline{\rho})(\cdot, t) \in H^{1}(\mathbb{R})$ at least and thus $(1/\rho_{n}) (t,x) \to 1/\overline{\rho}$ as $|x| \to \infty$. 
    It follows that $1/\rho_{n}$ attains its maximum at some point $x_{t} \in \mathbb{R}$ for any $t \in [0,T^{*})$ and so \eqref{PnMAX} is well-defined. 
    As in \cite{HCLMehmood} we can also show that $(P_{n}^{M})'(t) = \partial_{t}P_{n}(x_{t},t)$. Substituting $u_{n} = w_{n} - \partial_{x}p_{n}(\rho_n)$ into \eqref{1rhoeqn} and evaluating the equation at the minimum points of $1/\rho_{n}$, we eventually get
    \begin{equation}
    \partial_{t}P_{n}^{M} = - u_{n} \partial_{x} P_{n}^{M} + \rho_{n}^{-1} \partial_{x}w_{n} - P_{n}^{M} p_{n}''(\rho_{n})|\rho_{n}^{2} \partial_{x}P_{n}^{M}|^{2} - 2p_{n}'(\rho_{n})\rho_{n}^{2} \left( \partial_{x} P_{n}^{M}  \right)^{2} + \rho_{n} p_{n}'(\rho_{n}) \partial_{x}^{2}P_{n}^{M}.
    \end{equation} 
    Since $\rho_{n}(\cdot, t) > 0$ on $[0,T^{*})$, $\partial_{x}P_{n}^{M}(t) = 0 $ and $ \partial_{x}^{2}P_{n}^{M} \le 0$, we have that
    \begin{equation} \label{1rhoODE}
        (P_{n}^{M})'(t) \le \frac{\partial_{x}w_{n}(x_{t},t)}{\rho_{n}(x_{t},t)} ~~ \forall ~t \in [0,T^{*}).
    \end{equation} Using \eqref{Wn_cons} we infer that
    \begin{equation} \label{lbdensity}
       \rho_{n}(t,x) \ge \frac{1}{M_{n}^{0}t + (\inf_{\mathbb{R}}\rho_{n}^{0})^{-1}}, \text{ for a.e. } t \in [0,T^{*}).
    \end{equation} 
    The lower bound \eqref{lbdensity} contradicts \eqref{buassumption} and therefore we conclude that $T^{*} = + \infty$.
\end{proof} 
The above proof grants us a lower bound on the density for our global solution.
\begin{corollary}
    The density $\rho_{n}$ provide by Theorem~\ref{localexistence} satisfies
    \begin{equation}
          \rho_{n}(t,x) \ge \frac{1}{M_{n}^{0}\, t + (r^0_n)^{-1}} \quad \text{ on } ~  [0,\infty)\times \mathbb{R}.
    \end{equation}
\end{corollary}
\section{Uniform estimates} \label{sectionUNIF}
In this section we obtain a series of uniform in $n$ estimates which we will need in the limit passage.
\subsection{Estimates for the momentum}
We now look for uniform bounds on the quantities $\rho_{n}u_{n}$ and $\rho_{n}w_{n}$. 
We first obtain $L^\infty L^1$ bounds deriving naturally from the conservative formulation of the equations. 
\begin{lemma}
    The momentum $\rho_{n}u_{n}$ satisfies\begin{equation} \label{rhouL1}
        \|\rho_{n}u_{n}\|_{L^{\infty}_{t}L^{1}_{x}} \le \|\rho_{n}^{0}u_{n}^{0}\|_{L^{1}_{x}}.
    \end{equation}
\end{lemma}
\begin{proof}
    Suppose $S \in C^{2}(\mathbb{R})$ is an arbitrary convex function. Multiplying \eqref{a2} by $S'(u_{n})$, one can show that
    \begin{equation}
        \partial_{t}(\rho_{n}S(u_{n})) + \partial_{x}(\rho_{n}S(u_{n})) - \partial_{x}(S'(u_{n})\lambda_{n}(\rho_{n})\partial_{x}u_{n}) = - S''(u_{n})\lambda_{n}(\rho_{n})(\partial_{x}u_{n})^{2} \le 0. 
    \end{equation} Integrating in space and time leads to
    \begin{equation} \label{rhouapprox}
        \int_{\mathbb{R}}\rho_{n}S(u_{n})(t,x)~dx \le \int_{\mathbb{R}}\rho_{n}^{0}S(u_{n}^{0})(x)~dx.
    \end{equation} We choose $S = \varphi_{\alpha},$ where $ \varphi_{\alpha}(x) = \sqrt{x^{2} + \alpha}, ~\alpha >0$. Note that for each $\alpha > 0$, $\varphi_{\alpha} > 0$ is a convex function belonging to $C^{2}(\mathbb{R})$. In particular, $\varphi_{\alpha}(\cdot) \to |\cdot|$ uniformly as $\alpha \to 0$. Using Fatou's lemma in \eqref{rhouapprox} leads to \eqref{rhouL1}.
\end{proof}
\begin{lemma}
    The desired momentum $\rho_{n}w_{n}$ satisfies
    \begin{equation} \label{rhowL1}
        \|\rho_{n}w_{n}\|_{L^{\infty}_{t}L^{1}_{x}} \le \|\rho_{n}^{0}w_{n}^{0}\|_{L^{1}_{x}}.
    \end{equation}\begin{proof}
        We use the same argument as in the previous lemma. Multiplying \eqref{AR-2} by an arbitrary convex $S \in C^{2}(\mathbb{R})$ and rearranging, we get
        \begin{equation}
            \partial_{t}(\rho_{n}S(w_{n})) + \partial_{x}(\rho_{n}u_{n}S(w_{n})) = 0.
        \end{equation} Integrating in space and time, choosing $S = \varphi_{\alpha}$ and using Fatou's lemma once more gives us \eqref{rhowL1}.
    \end{proof}
\end{lemma}
This directly leads to a bound on $\partial_{x}\pi_{n}$.
\begin{corollary} \label{dxpilemma}
Under the assumption \eqref{A2mom}, the potential $\pi_{n}$ satisfies
\begin{equation} \label{dxpiL1}
        \|\partial_{x}\pi_{n}\|_{L^{\infty}_{t}L^{1}_{x}} \le C.
    \end{equation}
\end{corollary}
\begin{proof}
    This is an immediate consequence of the relationship $\rho_{n}w_{n} = \rho_{n}u_{n} + \partial_{x}\pi_{n}$.
\end{proof}
\subsection{Estimates for the velocity \texorpdfstring{$u_n$}{Lg}}
We can use a maximum-principle argument to show that $u_{n}$ is uniformly bounded.
\begin{lemma}
    The actual velocity $u_{n}$ satisfies
    \begin{equation} \label{unLinf}
    \|u_{n}\|_{L^{\infty}_{t,x}} \le C.
    \end{equation}
\end{lemma}
\begin{proof} The momentum equation \eqref{a2} can be expressed as
\begin{equation} \label{momu}
    \partial_{t}u_{n} = - u_{n}\partial_{x}u_{n} + \rho_{n}^{-1}\partial_{x}(\lambda_{n}(\rho_{n})\partial_{x}u_{n}).
\end{equation}
    Fix a time $t \in [0,T]$. Notice that since $u_{n} \in L^{2}(\mathbb{R})$ is uniformly continuous, we must have that $u_{n}(t,x) \to 0$ as $|x| \to \infty$. Our first goal is to show that \begin{equation} \label{unMAXabove}
    u_{n}(t,x) \le \max\left(0, ~\esssup_{x \in \mathbb{R}}(u_{n}^{0})\right)  ~~ \text{ for all } (t,x) \in  [0,T]\times\mathbb{R} .\end{equation} First suppose that $\esssup_{x \in \mathbb{R}}u_{n}(x, t) \le 0$. Then \eqref{unabove} is satisfied trivially. Therefore we may assume without loss of generality that $\esssup_{x \in \mathbb{R}}u_{n}(x, t) > 0$.  This implies that $u_{n}(\cdot, t)$ obtains its maximum, i.e. there exists $x_{t} \in \mathbb{R}$ such that $\esssup_{x \in \mathbb{R}} u_{n}(t,x) = u_{n}(x_{t},t)$. This means that the maximum function \begin{equation}
        u_{n}^{M}(t) := \max_{x \in \mathbb{R}} u_{n}(t,x)
    \end{equation} is well-defined, almost everywhere differentiable and satisfies $(u_{n}^{M})'(t) = \partial_{t}u_{n}(x_{t},t)$. Evaluating \eqref{momu} at the points $(x_{t},t)$ then gives us
\begin{equation}
    \partial_{t}u_{n}(x_{t},t) = -u_{n}\partial_{x}u_{n} + \rho_{n}^{-1} \partial_{x}\lambda_{n}\partial_{x}u_{n} + \rho_{n}^{-1}\lambda_{n}\partial_{x}^{2}u_{n} \le 0,
\end{equation} where we have used the facts $\partial_{x}u_{n}(x_{t},t) = 0$ and $ \partial_{x}^{2}u_{n}(x_{t},t) \le 0$. Thus we get \begin{equation}
    u_{n}(t,x) \le \max_{x \in \mathbb{R}} u_{n}^{0}(x) = \esssup_{x \in \mathbb{R}} u_{n}^{0}(x) ~~ \text{ for all } (t,x) \in [0,T] \times \mathbb{R}, \label{unabove}
\end{equation} under the assumption that $\sup_{x \in \mathbb{R}}u_{n}(x, t) > 0$. In particular we obtain \eqref{unMAXabove}. Repeating this argument with the minimum points instead of the maximum points will give us \begin{equation}
         u_{n}(t,x) \ge \min\left(0, ~\esssinf_{x \in \mathbb{R}}(u_{n}^{0})\right) ~~ \text{ for all } (t,x) \in [0,T] \times \mathbb{R}. \label{unbelow}
    \end{equation} Taking into account \eqref{unMAXabove}, \eqref{unbelow} and the assumption \eqref{A4uLinf}, we obtain \eqref{unLinf}.
\end{proof} 
Our next goal is to derive the one-sided Lipschitz condition for $u_{n}$. We will first prove an intermediate result for which we introduce the singular diffusion $V_{n}$ as 
    \begin{equation}
        V_{n} := \lambda_{n}(\rho_{n})\partial_{x}u_{n}.
    \end{equation} This corresponds to the active potential used by Constantin et al in \cite{Constantin_2020}. The following lemma is a type of maximum-principle for $V_{n}$.
\begin{lemma} We have
\begin{equation} \label{VnONESIDE}
    (\lambda_{n}(\rho_{n})\partial_{x}u_{n})(t,x) \le \esssup_{x \in \mathbb{R}}(\lambda_{n}(\rho_{n}^{0})\partial_{x}u_{n}^{0}).
\end{equation}
\end{lemma}
\begin{proof}
    We define the singular-diffusion $V_{n}$ as 
    \begin{equation}
        V_{n} = \lambda_{n}(\rho_{n})\partial_{x}u_{n}.
    \end{equation} This corresponds to the active potential used by Constantin et al \cite{Constantin_2020}. The evolution equation satisfied by $V_{n}$ is given by (see \cite{HCL, HCLMehmood})
    \begin{align}
    \partial_{t} V_{n} + \left(u_{n} + \frac{\lambda_{n}(\rho_{n})}{\rho_{n}^{2}} \partial_{x}\rho_{n} \right) \partial_{x}V_{n} - \frac{\lambda_{n}(\rho_{n})}{\rho_{n}} \partial_{x}^{2}V_{n} = - \frac{(\lambda_{n}'(\rho_{n})\rho_{n} + \lambda_{n}(\rho_{n}))}{(\lambda_{n}(\rho_{n}))^{2}}V_{n}^{2}. \label{Veqn}
\end{align} 
Next, we introduce the maximum function 
\begin{equation} \label{Vnmax}
    V_{n}^{M}(t) := (\lambda_{n}(\rho_n)\partial_{x}u_{n})(t,x_{t}),
\end{equation} 
where $x_{t}$ is the point where $\lambda_{n}\partial_{x}u_{n}(\cdot,t)$ obtains its maximum. 
To prove \eqref{VnONESIDE} it is sufficient to show that
\begin{equation} \label{VnCONS}
    V_{n}^{M}(t) \le V_{n}^{M}(0)
\end{equation}We need to justify that \eqref{Vnmax} is well-defined. First note that $V_{n}(\cdot, t)$ is a uniformly continuous function belonging to $L^{2}(\mathbb{R})$. Therefore, assuming that $\esssup_{x \in \mathbb{R}} V_{n}(t,x) > 0$, we can say that $V_{n}(\cdot, t)$ attains its supremum for any $t \in [0,T]$ fixed. In fact, since $\rho_{n} > 0$, it is true that $\esssup_{x \in \mathbb{R}} V_{n}(t,x) \ge 0$. This is because $\partial_{x}u_{n}$ is uniformly continuous and cannot be negative everywhere since this would contradict the fact that $u_{n} \in H^{1}(\mathbb{R})$. The case $\esssup_{x \in \mathbb{R}} V_{n}(t,x) = 0$ can only occur if $u_{n}(\cdot, t)$ is constant, in which case $V_{n}(\cdot, t) \equiv 0$ and so $V_{n}(\cdot, t)$ trivially satisfies \eqref{VnCONS}. As a result we may assume without loss of generality that $\esssup_{x \in \mathbb{R}} V_{n}(t,x) > 0$, in which case \eqref{Vnmax} is well-defined. Evaluating \eqref{Veqn} at $(x_{t},t)$ and arguing as in the previous lemma, we arrive at $\partial_{t}V_{n}^{M}(t) \le 0$ which implies \eqref{VnCONS}. 
 \end{proof}

\begin{corollary}
    Assuming the condition \eqref{A3A} on the initial data, the velocity $u_{n}$ satisfies \begin{equation} \label{dxuONESIDE}
        \partial_{x}u_{n} \le C.
    \end{equation}
\end{corollary} \begin{proof}
    Recall that $\lambda_{n}(\rho_{n}) = \gamma_{n} \rho_{n}^{\gamma_{n}+1}$. The lower bound \eqref{lbdensity} implies that
\begin{equation}
    \frac{1}{\lambda_{n}(\rho_{n})} \le \frac{\big(M_n^0T + (r^0_n)^{-1}\big)^{\gamma_{n}+1}}{\gamma_{n}} ~~ \text{ on } [0,T] \times \mathbb{R}.
\end{equation} Going back to \eqref{VnONESIDE}, the assumption \eqref{A3A} yields that
\begin{equation}
    \partial_{x}u_{n} \le  \frac{C\cdot \gamma_{n} \big(M_n^0T + (r^0_n)^{-1}\big))^{\gamma_{n}+1}}{\gamma_{n}\big(M_n^0T + (r^0_n)^{-1}\big)^{\gamma_{n}+1}} = C.
\end{equation}
\end{proof} This leads to a local integrability estimate for $\partial_{x}u_{n}$.
\begin{corollary}
    For any compact $K \subset \mathbb{R}$ the velocity $u_{n}$ satisfies
    \begin{equation} \label{dxuL1est}
        \|\partial_{x}u_{n}\|_{L^{\infty}(0,T; L^{1}(K))} \le C_{K},
    \end{equation} where $C_{K} > 0$ is a constant depending on the compact set $K$ but independent of $n$.
\end{corollary}
\begin{proof}
     Adopting the notation $(f)_{+} := \max(0, f)$ for an appropriate function $f$, we have the decomposition \begin{equation} \label{pmdecomposition}
        |f| = 2(f)_{+} - f.
    \end{equation} Fix a compact set $K \subset \mathbb{R}$. For simplicity we can assume $K = [-M,M]$ for some $M>0$. Then using \eqref{pmdecomposition},
    \begin{align*}
        \int_{K} |\partial_{x}u_{n}|~dx &=  2 \int_{K} (\partial_{x}u_{n})_{+}~dx - \int_{K} \partial_{x}u_{n}~dx \\[1ex] &\le 2C_{K} - \left( u_{n}(M,t) - u_{n}(-M, t) \right) \\[1ex] &\le C_{K},
    \end{align*} where we have used \eqref{dxuONESIDE} and \eqref{unLinf}.
\end{proof}
\subsection{Estimates for the potential \texorpdfstring{$\pi_{n}$}{Lg}}
Due to the definition~\eqref{df:pi_n}, we have
\begin{equation} \label{piexact}
   \pi_n = \pi_{n}(\rho_{n}) = \frac{\gamma_{n}}{\gamma_{n}+1}\rho_{n}^{\gamma_{n}+1}.
\end{equation}

The aim of this subsection is to prove the following result.
\begin{lemma} 
     For any compact $K \subset \mathbb{R}$, there exists a positive constant $C_K$ independent of $n$, such that the potential $\pi_{n}$ satisfies 
     \begin{equation} \label{pilemma}
        \|\pi_{n}\|_{L^{\infty}(0,T; L^{1}(K))} \le C_{K}.
     \end{equation}
\end{lemma}
\begin{proof}
    We fix an arbitrary positive $\phi \in \mathcal{D}(\mathbb{R})$ and define the test function \begin{equation}
        \psi(t,x) := \int_{-\infty}^{x} \phi(y)~dy.
    \end{equation} Note that $\|\psi\|_{L^{\infty}_{t,x}} \le C$. Multiplying the momentum equation \eqref{a2} by $\psi$ and integrating by parts in space and time, we get
    \begin{align*}
        \int_{\mathbb{R}} \psi(\rho_{n}u_{n}(t,x) - \rho_{n}^{0}u_{n}^{0}(x))~dx  - \int_{0}^{t} \int_{\mathbb{R}} \phi \rho_{n}u_{n}^{2}~dxds = - \int_{0}^{t} \int_{\mathbb{R}} \phi \lambda_{n}(\rho_{n})\partial_{x}u_{n}~dxds.
    \end{align*} Using \eqref{rhouL1}, \eqref{unLinf} and the assumption \eqref{A2mom}, we get
    \begin{equation} \label{VnPHI}
        \left| \int_{0}^{t} \int_{\mathbb{R}} \phi \lambda_{n}(\rho_{n})\partial_{x}u_{n}~dxds \right| \le C.
    \end{equation} Next, the evolution equation for $\pi_{n}$ is given by
    \begin{equation} \label{eqnpin}
        \partial_{t}\pi_{n} + u_{n} \partial_{x}\pi_{n} + \lambda_{n}(\rho_{n})\partial_{x}u_{n} = 0.
    \end{equation} Multiplying by $\phi$ and integrating,
    \begin{equation}
        \int_{\mathbb{R}} \phi \pi_{n}(t,x) - \phi \pi_{n}(x,0)~dx + \int_{0}^{t} \int_{\mathbb{R}} \phi u_{n} \partial_{x}\pi_{n}~dxds + \int_{0}^{t} \int_{\mathbb{R}}\phi \lambda_{n}(\rho_{n})\partial_{x}u_{n}~dxds = 0.
    \end{equation} Using \eqref{VnPHI}, \eqref{unLinf}, \eqref{dxpiL1} and the assumption \eqref{densityposassumption} we get
    \begin{equation}
        \int_{\mathbb{R}} \phi \pi_{n}(t,x)~dxds \le C,
    \end{equation} which yields the required bound.
\end{proof}
\begin{corollary}
    For any compact $K \subset \mathbb{R}$ the density $\rho_{n}$ satisfies
    \begin{equation} \label{rhoLINFlocal}
        \|\rho_{n}\|_{L^{\infty}(0,T; L^{\infty}(K))} \le C_{K}.
    \end{equation}
\end{corollary}
\begin{proof}
   We already know that $\rho_{n} > 0$ from \eqref{lbdensity}. Using \eqref{dxpiL1}, \eqref{pilemma} and the Sobolev embedding $W^{1,1}(\mathbb{R}) \hookrightarrow L^{\infty}(\mathbb{R})$ we get 
   $\|\pi_{n}\|_{L^{\infty}(0,T; L^{\infty}(K))} \le C_{K}$. 
   Recalling \eqref{piexact} this implies
   \begin{equation} \label{rhonUPPER}
       \rho_{n}^{\gamma_{n}+1} \le C_{K} ~~ \text{ on } K \times [0,T],
   \end{equation} 
   for any compact $K \subset \mathbb{R}$, which gives us \eqref{rhoLINFlocal}.
\end{proof} 
\begin{remark}
    For $n$ fixed, \eqref{dxpiL1} and \eqref{pilemma} imply that the restriction of $\pi_{n}$ to a compact set $K$ is bounded by a constant which depends on $K$. But rooting both sides of \eqref{rhonUPPER}, we find that $\rho_{n}$ is uniformly bounded by a constant which converges to $1$ as $n \to \infty$. This means that upon passing to the limit in $n$, our limit density $\rho$ will be bounded by $1$ on any compact $K$ which actually implies that $0 \le \rho \le 1$ a.e. on $\mathbb{R}$. Thus, although we only obtain a local bound on $\rho_{n}$ for fixed $n$, this transforms into a global bound for $\rho$ in the limit.
\end{remark}

In summary, we have shown that under the assumptions of Theorem \ref{existencedual} there exists $C>0$ independent of $n$ with
\begin{equation}
     \|\rho_{n}u_{n}\|_{L^{\infty}_{t}L^{1}_{x}} + \|\rho_{n}w_{n}\|_{L^{\infty}_{t}L^{1}_{x}}  \le C, \label{dualbounds}
\end{equation} 
and for any compact $K$, there exists $C_K > 0$ such that
\begin{equation}
    \|u_{n}\|_{L^{\infty}(0,T; W^{1,1}(K))} + \|\pi_{n}\|_{L^{\infty}(0,T; W^{1,1}(K))} +  \|\rho_{n}\|_{L^{\infty}((0,T) \times K)} \le C_K. \label{dualbounds-bis}
\end{equation} 

\section{The limit passage}{\label{sec:limit}}
In this section we complete the proofs of Theorems \ref{limitexistence} and \ref{existencedual}.
\subsection{Existence of duality solutions to the hard-congestion model}
 \begin{proof}[Proof of Theorem \ref{existencedual}] 
 We first make note of a key result which is mentioned without proof in Remark 4.2.4. of \cite{bouchut1998one}. We provide a proof in Section~\ref{sec:final}.
\begin{proposition} \label{weakimpliesduality}
    Suppose $f \in C([0,T]; L^{1}_{loc,w}(\mathbb{R}))$ is a weak solution to $\partial_{t} f + \partial_{x}(af) = 0$ on $(0,T) \times \mathbb{R}$ in the sense that for any $t \in [0,T]$,
    \begin{equation} \label{weakcont}
         \int_{0}^{t}\int_{\mathbb{R}} f \partial_{t}\phi + f u \partial_{x} \phi ~dxds = \int_{\mathbb{R}} f\phi(t,x) - f\phi(0,x)~dx ~~ \forall ~ \phi \in W^{1,\infty}([0,T] \times \mathbb{R}),
    \end{equation}
    where $a \in L^{\infty}((0,T) \times \mathbb{R})$ satisfies the OSL condition \eqref{OSLC}. Then $f$ is also a duality solution.
\end{proposition}
We first deal with the sequence $\rho_{n}$. The bound \eqref{rhoLINFlocal} implies that there exists $\rho \in L^{\infty}((0,T) \times \mathbb{R})$ such that up to a subsequence, \begin{equation}
    \rho_{n} \rightharpoonup^{\star} \rho ~~ \text{ in } L^{\infty}((0,T) \times \mathbb{R}). 
\end{equation} Indeed, \eqref{rhoLINFlocal} implies that $\rho \le 1$ on any compact $K \subset \mathbb{R}$ and therefore $0 \le \rho \le 1$. Next, the estimate $\|\rho_{n}u_{n}\|_{L^{\infty}_{t}L^{1}_{x}} \le C$ (see \eqref{rhouL1}) and the continuity equation \eqref{HCL-n1} imply that \begin{equation} \label{dtrho}
\|\partial_{t}\rho_{n}\|_{L^{\infty}_{t}W^{-2,2}_{x}} \le C.
    \end{equation} We now recall the following compensated compactness result.
    \begin{lemma}[Lemma 5.1, \cite{mathlions}] \label{lionsCC}
Let $g_{n}, h_{n}$ converge weakly to $g,h$ respectively in $L^{p_{1}}(0,T; L^{p_{2}}(\Omega))$ and $ L^{q_{1}}(0,T; L^{q_{2}}(\Omega))$ where $1 \le p_{1}, ~ p_{2} \le +\infty$,
\begin{equation*}
    \frac{1}{p_{1}} + \frac{1}{q_{1}} = \frac{1}{p_{2}} + \frac{1}{q_{2}} = 1.
\end{equation*} Assume in addition that\begin{itemize}
    \item \normalfont{(A1):} $\partial_{t}g_{n}$ is bounded in $L^{1}(0,T; W^{-m,1}(\Omega))$ for some $m \ge 0$ independent of $n,$
    \item \normalfont{(A2):} $\|h_{n} - h_{n}(\cdot + \zeta, t)\|_{L^{q_{1}}(0,T;L^{q_{2}}(\Omega))} \to 0$ as $|\zeta| \to 0,$ uniformly in $n$.  \vspace{5pt}
\end{itemize} Then $g_{n}h_{n}\longrightarrow gh$ in $\mathcal{D}'( (0,T) \times \Omega)$.
\end{lemma}
    Applying this lemma once with $(g_{n}, h_{n}) = (\rho_{n}, u_{n})$ and again with $(g_{n}, h_{n}) = (\rho_{n}u_{n}, u_{n})$ yields
\begin{equation} \label{rhouCONV}
    \rho_{n}u_{n} \to \rho u, ~ \rho_{n}u_{n}^{2} \to \rho u^{2} ~~ \text{ in } \mathcal{D}'((0,T) \times \mathbb{R}).
\end{equation} This is enough to verify that the limit $\rho$ satisfies the continuity equation in the distributional sense with $0 \le \rho \le 1$. Our next task is to prove the existence of a duality solution to the momentum equation. To this end, we define $m_{n} := \rho_{n}w_{n}$. For each $n$ we know that $m_{n}$ is (at least) a $C([0,T]; L^{\infty}(\mathbb{R}))$ solution to \eqref{AR-2}. In particular, by Proposition \ref{weakimpliesduality} they are duality solutions. Using the hypotheses of Theorem \ref{existencedual}, we can verify that the assumptions of the stability result (the third item in Theorem \ref{DUALITYresults}) are met. As a consequence, there exists $m \in \mathcal{S}_{\mathcal{M}}$ such that  $m_{n} \to m \text{ in } \mathcal{S}_{\mathcal{M}}$ where $m$ solves
\begin{equation} \label{mDUAL}
    \begin{cases}
 \partial_{t}m+\partial_{x}(m u)=0 & \text{ in }(0,T) \times \mathbb{R}, \\
m(x,0) = m^{0}(x) \in \mathcal{M}_{loc}(\mathbb{R}),
\end{cases}
\end{equation} in the duality sense. Next, we need to verify the switching relation  \begin{equation} \label{switchingrelation}
    (1-\rho)\pi = 0 ~ \text{ a.e. in } [0,T] \times \mathbb{R}.  
\end{equation} This can be done in a very similar way to \cite{HCLMehmood}. 
Firstly,
\begin{align*}
    (1-\rho_{n})\pi_{n} &= \frac{\gamma_{n}}{\gamma_{n}+1}(1-\rho_{n}) \rho_{n}^{\gamma_{n}+1} \\[1ex]
    &= \frac{\gamma_{n}}{\gamma_{n}+1} \left((1-\rho_{n})  \rho_{n}^{\gamma_{n}+1} \mathbbm{1}_{\left\{ (t,x) ~ : ~ 0 < \rho_{n}(t,x) \le 1  \right\}} + (1-\rho_{n})  \rho_{n}^{\gamma_{n}+1} \mathbbm{1}_{\left\{ (t,x) ~ : ~  \rho_{n}(t,x) > 1  \right\}} \right) \\[1ex] &=: A_{n} + B_{n}.
\end{align*}
It is straightforward to see that $A_{n} \to 0$ a.e. as $n \to \infty$. From \eqref{rhonUPPER} we infer that the Lebesgue measure of the set $\{\rho_{n} > 1\}$ goes to $0$ as $n \to \infty$. Therefore using \eqref{weakbounds}, we can justify that
\begin{align*}
    \|B_{n}\|_{L^{1}_{t,x}} &= \int_{0}^{T} \int_{K} |(1-\rho_{n})\pi_{n}| \mathbbm{1}_{\left\{ (t,x) ~ : ~  \rho_{n}(t,x) > 1  \right\}} ~dxds \\[1ex] &\le \|1-\rho_{n}\|_{L^{\infty}_{t,x}}\|\pi_{n}\|_{L^{1}_{t}L^{1}_{loc,x}} \cdot \mu\left( \left\{ \rho_{n} > 1 \right\} \right)  \to 0.
\end{align*} Therefore we have $ \|(1-\rho_{n})\pi_{n}\|_{L^{1}(0,T; L^{1}_{loc}(\mathbb{R}))} \to 0$. On the other hand, since  we have $\|\pi_{n}\|_{L^{\infty}_{t}W^{1,1}_{x}} + \| \rho_{n}\|_{L^{\infty}_{t,x}} + \|\partial_{t}\rho_{n}\|_{L^{\infty}_{t}W^{-2,2}_{x}} \le C$ locally in space, another application of Lemma 5.1. from \cite{mathlions} gives us
\begin{equation} \label{1-rhodistr}
    (1-\rho_{n})\pi_{n} \longrightarrow (1-\rho)\pi, ~ \text{ in } \mathcal{D}'((0,T) \times \mathbb{R}).
\end{equation} 
This allows us to conclude that $(1-\rho)\pi = 0$ a.e. in $[0,T] \times K$ for any compact $K \subset \mathbb{R}$ and hence also a.e. in $[0,T] \times \mathbb{R}$.

With this we have now established parts (i)-(iii) from Definition \ref{defndualityHCL}. It remains to prove the non-linear coupling. For fixed $n$ we know from the relation $w_{n} = u_{n} + \partial_{x}p_{n}(\rho_n)$ that $m_{n} = \rho_{n}u_{n} + \partial_{x}\pi_{n}$.
Passing to the limit $n \to \infty$ and using \eqref{rhouCONV}, we get \begin{equation}
    m = \rho u + \partial_{x} \pi ~~ \text{ in the sense of measures, }
\end{equation} with $\partial_{x}\pi \in \mathcal{M}((0,T) \times \mathbb{R})$ being the distributional derivative of $\pi$. Thus $(\rho, m, \pi)$ is a duality solution to the limit system. The entropy inequality \eqref{dualityentropy} follows directly from the properties of duality solutions, namely Theorem 4.3.6 of \cite{bouchut1998one}. The OSL condition is also verified since $\partial_{x}u_{n} \le C$ still holds under the assumptions of Theorem \ref{existencedual}. This concludes the proof of Theorem \ref{existencedual}. \end{proof}
\subsection{Existence of weak solutions to the hard-congestion model}
\begin{proof}[Proof of Theorem \ref{limitexistence}]
    Recall that for $n$ fixed we proved the existence of a regular solution $(\rho_{n}, u_{n})$ to \eqref{a1}-\eqref{a2}, which can be re-expressed as
\begin{subequations} \label{HCL-n}
\newcommand{\mystrut}{\vphantom{\pder{}{}}}
\begin{numcases}{}
    \partial_{t} \rho_{n} + \partial_{x} (\rho_{n} u_{n}) = 0, ~~ \text { on } (0,T) \times \mathbb{R}, \label{HCL-n1}   \\[1ex]
    \partial_{t} (\rho  u + \partial_{x}\pi) + \partial_{x}((\rho u + \partial_{x}\pi)u)  = 0,  ~~ \text { on } (0,T) \times \mathbb{R}. \label{HCL-n2}
\end{numcases}
\end{subequations} 
 We now obtain a stronger bound for $\partial_{x}\pi_{n}$. 
\begin{corollary}
   Under the assumptions \eqref{A2mom} and \eqref{A5w}, we have for any compact $K \subset \mathbb{R}$ that
    \begin{equation} \label{dxpiL2}
        \|\partial_{x}\pi_{n}\|_{L^{\infty}(0,T; L^{2}(K))} \le C_{K}.
    \end{equation}
\end{corollary}
\begin{proof}
    Recall from \eqref{rhowENERGYest} that $\|\sqrt{\rho_{n}}w_{n}\|_{L^{\infty}_tL^2_x} \le C$. Fixing a compact $K \subset \mathbb{R}$, we have
    \begin{align*}
        \int_{K} |\partial_{x}\pi_{n}|^{2}~dx
        &\le 2 \left[ \int_{K} \rho_{n}^{2}u_{n}^{2} ~dx + \int_{K} \rho_{n}^{2}w_{n}^{2} ~dx \right] \\[1ex] 
        &\le \|\rho_{n}u_{n}\|_{L^{\infty}_{t}L^{\infty}_{loc, x}} \|\rho_{n}u_{n}\|_{L^{\infty}_{t}L^{1}_{x}} + \|\rho_{n}\|_{L^{\infty}_{t, x}} \|\sqrt{\rho_{n}}w_{n}\|_{L^{\infty}_{t}L^{2}_{x}} \\[1ex] 
        &\le C_{K}.
    \end{align*}
\end{proof}
Thus under the assumptions of Theorem \ref{limitexistence} we have the bounds \begin{equation}
    \|u_{n}\|_{L^{\infty}(0,T; W^{1,1}(K))} + \|\rho_{n}u_{n}\|_{L^{\infty}_{t}L^{1}_{x}} + \|\rho_{n}w_{n}^{2}\|_{L^{\infty}_{t}L^{1}_{x}} + \|\pi_{n}\|_{L^{\infty}(0,T; W^{1,2}(K))} +  \|\rho_{n}\|_{L^{\infty}((0,T) \times  K)} \le C_K. \label{weakbounds}
\end{equation} 
These bounds imply that there exists a triple $(\rho, u, \pi)$ such that up to a subsequence,
\begin{align}                   \label{RHOW}
    &\rho_{n} \rightharpoonup^{*} \rho \text{ weakly-* in }L^{\infty}( (0,T) \times K), \\[1ex]        \label{WWSTAR}
    &u_{n} \rightharpoonup^{*} u \text{ weakly-* in } L^{\infty}( (0,T) \times \mathbb{R}), \\[1ex] \label{PIWSTAR}
    &\pi_{n} \rightharpoonup^{*} \pi \text{ weakly-* in }L^{\infty}((0,T) \times K), \\[1ex]
    \label{DXPIWSTAR}
    &\partial_{x}\pi_{n} \rightharpoonup \partial_{x} \pi \text{ weakly in } L^{2}(0,T; L^{2}(K)),
\end{align} for any compact $K \subset \mathbb{R}$. Our next goal is to obtain a bound on $\partial_{t}\partial_{x}\pi_{n}$.
To this end, we acquire a local estimate on $\lambda_{n}\partial_{x}u_{n}$. Using the decomposition \eqref{pmdecomposition}, we can argue that for any $\phi \in \mathcal{D}(\mathbb{R})$, 
\begin{equation}
    \int_{0}^{t}\int_{\mathbb{R}} |\phi \lambda_{n}\partial_{x}u_{n}|~dxds = 2\int_{0}^{t}\int_{\mathbb{R}} (\phi \lambda_{n}\partial_{x}u_{n})_{+}~dxds - \int_{0}^{t}\int_{\mathbb{R}} \phi \lambda_{n}\partial_{x}u_{n}~dxds.
\end{equation}
Using \eqref{VnPHI} and \eqref{VnONESIDE}, we get \begin{equation} \label{VnL1}
    \|\lambda_{n}(\rho_{n})\partial_{x}u_{n}\|_{L^{1}(0,T; L^{1}_{loc}(\mathbb{R}))} \le C.
\end{equation}
Now we differentiate in space the evolution equation for $\rho_{n}p_{n}(\rho_{n}) = \frac{\gamma_{n}+1}{\gamma_{n}}\pi_{n}(\rho_{n})$ which gives
\begin{equation} \label{dtdxpi}
    \partial_{t}\partial_{x}(\rho_{n}p_{n}(\rho_n)) + \partial_{x}^{2}(\rho_{n}p_{n}(\rho_n)u_{n}) + \partial_{x}(\lambda_{n}(\rho_n)\partial_{x}u_{n}) = 0.
\end{equation} 
Fixing an arbitrary $\phi \in W^{2,2}_{0}(\mathbb{R})$ with $\|\phi\|_{W^{2,2}_{0}(\mathbb{R})} = 1$, we now estimate the $W^{-2,2}(\mathbb{R})$ norm of $\partial_{t}\partial_{x}\pi_{n}$. Denoting by $(\cdot, \cdot)$ the $L^{2}(\mathbb{R})$ inner product, we have using \eqref{dtdxpi} that
\begin{align*}
    &\int_{0}^{t} \|\partial_{t}\partial_{x}\big(\rho_{n}p_{n}(\rho_n)\big)\|_{W^{-2,2}(\mathbb{R})}~ds \\
    &= -\int_{0}^{t} (\pi_{n}u_{n}, \phi'') - (\lambda_{n}(\rho_n)\partial_{x}u_{n}, \phi')~ds \\[1ex] 
    &\le T \|\pi_{n}u_{n}\|_{L^{\infty}_{t}L^{\infty}_{loc,x}} \|\phi''\|_{L^{1}(\mathbb{R})} + \|\lambda_{n}(\rho_n)\partial_{x}u_{n}\|_{L^{1}_{t}L^{1}_{loc,x}}\|\phi'\|_{L^{\infty}(\mathbb{R})} \le C,
\end{align*}
thanks to \eqref{VnL1} and \eqref{weakbounds}. Hence we have \begin{equation}
\|\partial_{t}\partial_{x}\pi_{n}\|_{L^{1}_{t}W^{-2,2}_{x}} \le C.
\end{equation}
Recalling \eqref{rhouCONV} and \eqref{RHOW}-\eqref{DXPIWSTAR}, we can pass to the limit in each term of \eqref{HCL-n1}-\eqref{HCL-n2} except for the final term of \eqref{HCL-n2}. In order to justify $u_{n} \partial_{x}\pi_{n} \to u \partial_{x} \pi$ in $\mathcal{D}'$, one might first attempt to apply Lemma \ref{lionsCC}. However, this will not work since if we take $(g_{n}, h_{n}) = (\partial_{x}\pi_{n}, u_{n})$ we do not have the weak-$\star$ convergence $\partial_{x}\pi_{n} \rightharpoonup \partial_{x}\pi$ in $L^{p}_{t}L^{\infty}_{x}$ for any $p \in [1, \infty]$. This convergence is required if we choose $q_{2} = 1$ (otherwise it is not clear if assumption (A2) is satisfied from \eqref{dxuL1est}). Nonetheless, we can use the following generalised compensated compactness result due to Moussa \cite{moussa2016some}:
\begin{lemma}[Proposition 3, \cite{moussa2016some}] Let $q \in [1, \infty], ~\alpha \in [1, +\infty)$, $T>0$ and $\Omega \subset \mathbb{R}$ a bounded open set with Lipschitz boundary. Consider two sequences $(a_{n})_{n}$ bounded in $L^{q}(0,T; W^{1,1}(\Omega))$ and $(b_{n})_{n}$ bounded in $L^{q'}(0,T; L^{\alpha'}(\Omega))$ respectively weakly or weak-$\star$ converging in these spaces to $a$ and $b$. If $(\partial_{t}b_{n})$ is bounded in $\mathcal{M}(0,T; H^{-m}(\Omega))$ for some $m \in \mathbb{N}$ then up to a subsequence we have the following weak-$\star$ convergence in $\mathcal{M}((0,T) \times \bar{\Omega})$ (i.e. with $C^{0}((0,T) \times \Omega)$ test functions):
\begin{equation}
    a_{n}b_{n} \longrightarrow ab,~~ \text{ as } n \to \infty.
\end{equation}
\end{lemma}
Taking any $q \in [1,\infty]$, $\alpha = 2$, $a_{n} = u_{n}$ and $b_{n} = \partial_{x}\pi_{n}$ we obtain the weak-$\star$ convergence $u_{n}\partial_{x}\pi_{n} \rightharpoonup u \partial_{x} \pi$ in $\mathcal{M}((0,T) \times \bar{\Omega})$ and therefore also in $\mathcal{D}'((0,T) \times \mathbb{R})$. It now remains to verify \eqref{HCL-L3}. First note that from the lower bound \eqref{lbdensity} and the upper bound $\rho_{n} \le C^{\frac{1}{\gamma_{n}+1}}$ seen in \eqref{rhoLINFlocal}, we must have $0 \le \rho \le 1$ almost everywhere in $K \times [0,T]$, for any compact $K \subset \mathbb{R}$. This implies $0 \le \rho \le 1$ a.e. on $\mathbb{R}$. Finally, the switching relation \eqref{switchingrelation} can be obtained in an identical fashion to the proof of Theorem \ref{existencedual}. This concludes the proof of Theorem \ref{limitexistence}. \end{proof}

\section{Final discussion on duality solutions}{\label{sec:final}}

\subsection{Comparing weak and duality solutions}
A-priori, duality solutions are not weak solutions since we cannot make sense of the product $m u$ in the distributional sense. It is natural to then wonder whether there is a relationship between the weak and duality solutions that we have obtained for the limit system. In this subsection we mention two results which help in understanding the connection between the two notions of solution. Firstly, recall that Proposition~\ref{weakimpliesduality} tells us that weak solutions to the continuity equation are duality solutions. We now provide a proof of this result.
\begin{proof}[Proof of Proposition \ref{weakimpliesduality}]
    Fix $\tau \in (0,T]$ and a reversible solution $p \in \text{Lip}((0,T) \times \mathbb{R})$ to \begin{equation} \label{reversibleeqn}
    \begin{cases}
 \partial_{t}p + a\partial_{x}p = 0 &\text{ in } (0, \tau) \times \mathbb{R}, \\
p(\tau, \cdot) = p^{\tau},
\end{cases}
\end{equation}where the final data $p^{\tau}$ is arbitrary. Taking an arbitrary $t \in (0,\tau]$ and $f=p$ in \eqref{weakcont} gives
\begin{equation}
     \int_{\mathbb{R}} f\phi(t,x) - f\phi(0,x)~dx = \int_{0}^{t}\int_{\mathbb{R}} f( \partial_{t}p +  u \partial_{x} p) ~dxds = 0.
\end{equation} Since $t$ was arbitrary we conclude that $t \mapsto \int_{\mathbb{R}} fp (t,x) ~dx$ is constant on $(0,\tau]$ and thus $f$ is a duality solution.
\end{proof}
We immediately deduce the following implication.
\begin{corollary}
    The weak solution to the hard-congestion model obtained in Theorem \ref{limitexistence} is also a duality solution.
\end{corollary}
\begin{proof}
    Defining $m := \rho u + \partial_{x} \pi$ where $(\rho, u, \pi)$ corresponds to the weak solution obtained in Theorem~\ref{limitexistence}, we have that $m$ is a weak solution to $\partial_{t}m + \partial_{x}(mu) = 0$ belonging to the space $C([0,T]; L^{2}_{w}(\mathbb{R}))$. Lemma~\ref{weakimpliesduality} implies that $m$ is also a duality solution to the same equation. Thus we obtain that $(\rho, m, \pi)$ is a duality solution to the hard-congestion model.
\end{proof}
The following result which is analogous to what can be seen in \cite{jamesaggregationvauch,bouchut1998one} provides some sort of equivalence between weak and duality solutions in the case where $u$ is piecewise continuous.
\begin{proposition}[\cite{jamesaggregationvauch, bouchut1998one}]
    Assuming that $u$ satisfies the OSL condition \eqref{OSLC} and is piecewise continuous on $(0,T) \times \mathbb{R}$ where the set of discontinuity is locally finite. Then there exists a function $\hat{u}$ which coincides with $u$ on the set of continuity of $u$. Additionally, $(\rho, m, \pi) \in \mathcal{S}_{L} \times \mathcal{S}_{\mathcal{M}} \times \mathcal{T}_{BV}$ with $0 \le \rho \le 1,~ \pi \ge 0$ is a duality solution to the hard-congestion model with velocity $u$ if and only if 
    \begin{equation} \label{dualHCLinD'}
    \begin{cases}
 \partial_{t}\rho + \partial_{x}(\rho \hat{u}) = 0 &\text{ in } \mathcal{D}', \\
\partial_{t}m + \partial_{x}(m \hat{u}) = 0 &\text{ in } \mathcal{D}', \\
 (1-\rho)\pi = 0 &\text{ a.e.}, \\
m = \rho \hat{u} + \partial_{x}\pi &\text{ in } \mathcal{M}.
\end{cases}
\end{equation}
 We also have that $u ~\triangle ~\rho = \hat{u} \rho$ and $u ~\triangle ~m = \hat{u} m$. In particular, $\hat{u}$ is a universal representative of $u$.
\end{proposition} This result essentially says that $(\rho, m, \pi)$ is a duality solution to the hard-congestion model with velocity $u$ if and only if $(\rho, \hat{u}, \pi) $ is a weak solution to the same model.

\subsection{(Non-)uniqueness of weak and duality solutions}{\label{sec:uniqueness}}
For a single conservation law with a fixed velocity $u$, uniqueness holds for duality solutions as long as $\alpha \in L^{1}(0,T)$ (see Theorem \ref{DUALITYresults}). This is a direct consequence of the definition of duality solutions. When dealing with a system of equations however, uniqueness is not so straightforward. In \cite{franccois1999duality}, Bouchut and James provided a definition for duality solutions to the system of pressureless gases and showed that they are also unique as long as $\alpha \in L^{1}(0,T)$. If $\alpha$ is not integrable at time $0$,
however, then uniqueness is lost for their system. 
Notice that their system is similar to ours since in the region $\{\rho < 1 \}$ our equations are formally reduced to a system of pressureless gases.

In contrast, we do not have uniqueness of weak or duality solutions to the hard-congestion model even for $\alpha \in L^{1}(0,T)$ and smooth initial data. We now describe one possible construction of a counterexample. Suppose that the velocity $u$ is constant in space, i.e. $u(t,x) = c(t)$ and additionally $c(0)=0$. Take an arbitrary $f \in BV_{loc}(\mathbb{R})$ such that $\min(f(x),f(x) - c(t)x) \ge 0$ on $[0,T]$.  Then considering the initial data $(\rho^{0}, m^{0}, \pi^{0}) = (1, f'(x),f(x))$ we have the solution
\begin{equation} \label{counterex1a} \textstyle
    (\rho_{1}, m_{1}, \pi_{1}; u_{1}) = (1,  ~ f'(x+\int_{0}^{t}c(s)ds), ~ f(x+ \int_{0}^{t}c(s)ds)-c(t)x; ~ c(t))
\end{equation} as well as
\begin{equation} \label{counterex1b} 
    (\rho_{2}, m_{2}, \pi_{2}; u_{2}) = (1, f'(x),~f(x);~ 0).
\end{equation}
\begin{remark}
It is a little easier to construct a counterexample to uniqueness on short time intervals. We simply need to assume $u=c(t)$ (with $c(0)=0$ and $c \ne 0$) is continuous and take any $f \in BV_{loc}(\mathbb{R})$ with $f > 0$. Then we have that \eqref{counterex1a} and \eqref{counterex1b} are both duality solutions (at least on some short time interval) with initial data $(\rho^{0}, m^{0}, \pi^{0}) = (1, f'(x),f(x))$.
\end{remark}
For a concrete example to non-uniqueness on a fixed time interval $[0,T]$, consider the initial data $(\rho^{0}, m^{0}, \pi^{0}) = (1, 2xe^{x^{2}}, e^{x^{2}}).$ We have the constant solution given by 
\begin{equation}
    (\rho_{1}, m_{1}, \pi_{1}; u_{1}) = (1, 2xe^{x^{2}},  e^{x^{2}};~ 0), \quad \text{ on } (0,T) \times \mathbb{R},
\end{equation} as well as
\begin{equation} \textstyle
    (\rho_{2}, m_{2}, \pi_{2}; u_{2}) = (1,  ~2(x+\frac{t^{2}}{2T})e^{(x+\frac{t^{2}}{2T})^{2}}, e^{(x+\frac{t^{2}}{2T})^{2}}+\frac{tx}{T};~ -\frac{t}{T}), \quad \text{ on } (0,T) \times \mathbb{R}.
\end{equation} The scaling of $1/T$ in the second example ensures $\pi_{1} \ge 0$ on $[0,T]$. Note that in the above examples, the system is initially fully congested, i.e. $\rho^{0} \equiv 1$. It is yet to be determined whether an example of non-uniqueness can be constructed for a more general initial density. In particular,  there does not seem to be a straightforward way to extend the above examples to cover this case. We postpone a more detailed discussion of uniqueness for a future work. To finish the section, we mention a simple criterion for uniqueness covering the case where congestion is not created or, if present initially, simply transported by the flow.
\begin{lemma}
    Let $(\rho, m, \pi)$ be a duality solution to the hard-congestion model with velocity $u$. Suppose we additionally have that $\rho u$ (or $\partial_{x}\pi$) is a duality solution to the conservation law $\partial_{t}f+\partial_{x}(fu)=0$ on $[0,T] \times \mathbb{R}$. Then $(\rho, m, \pi)$ is uniquely determined by the initial data, where $\pi$ is unique up to a constant.
\end{lemma}
\begin{proof} Suppose that $\rho u$ is a duality solution. Then $(\rho, q)$ solves the system of pressureless gases 
        \begin{equation} \label{BJsystem}
            \begin{cases}
                \partial_{t}\rho + \partial_{x}q = 0, \quad \text{ on } (0,T) \times \mathbb{R}, \\[1ex]
                 \partial_{t}q + \partial_{x}(q u) = 0, \quad \text{ on } (0,T) \times \mathbb{R}, \\[1ex]
                 u \rho = q,
            \end{cases}
        \end{equation} 
        in the duality sense of Bouchut and James \cite{franccois1999duality}. 
        Using the uniqueness result of \cite{franccois1999duality} (see Theorem 2.8, our initial density $\rho^0$ being non-atomic), we find that $(\rho, \rho u)$ is uniquely determined by the initial data $(\rho^{0}, m^{0}-\partial_{x}\pi^0)$. Additionally, since the space of duality solutions form a vector space, the relationship $m = \rho u + \partial_{x}\pi$ implies that $\partial_{x}\pi$ is also a duality solution to the problem
\begin{equation} \label{dxpiCONS}
    \begin{cases}
         \partial_{t}\partial_{x}\pi + \partial_{x}(u \partial_{x}\pi) = 0, \quad \text{ on } (0,T) \times \mathbb{R}, \\[1ex]
                \partial_{x}\pi(0,\cdot) = \partial_{x}\pi^{0} \in \mathcal{M}_{loc}(\mathbb{R}).
    \end{cases}
\end{equation}
Note that since $(\rho, \rho u)$ uniquely solves \eqref{BJsystem}, we have that $u$ is uniquely determined on the support of $\rho$ and in particular on the support of $\pi$ (and therefore $\partial_{x}\pi)$. Thus, the velocity $u$ associated with $\partial_{x}\pi$ is uniquely determined by the initial data $(\rho^{0}, m^{0}-\partial_{x}\pi^{0})$. Since uniqueness holds for duality solutions with a fixed velocity, we infer that $\partial_{x}\pi$ is also uniquely determined by the initial data and therefore so is $m = \rho u + \partial_{x}\pi$. On the other hand, suppose that $\partial_{x}\pi$ is a duality solution. Then so is $\rho u = m - \partial_{x}\pi$ and the pair $(\rho, q):= (\rho, \rho u)$ solves \eqref{BJsystem}.  Repeating the above argument leads to the desired conclusion.
\end{proof}

\begin{remark}
    Notice that this result does not hold true if we had defined the duality solution as a pair $(\rho, m)$ rather than a triple $(\rho, m, \pi)$. Indeed, consider the initial data $(\rho^{0}, m^{0})=(1, 2xe^{x^{2}})$. Then we have the trivial solution $(\rho_{1}, m_{1}; u_{1}, \pi_{1}) = (1,2xe^{x^{2}}; 0, e^{x^{2}})$ as well as \begin{equation}
      \textstyle   (\rho_{2}, m_{2}; u_{2}, \pi_{2}) = (1,2(x+\frac{t}{2T})e^{(x+\frac{t}{2T})^{2}}; ~-\frac{1}{2T},~ e^{(x+\frac{t}{2T})^{2}}+\frac{x}{2T}).
    \end{equation} This shows that even with a velocity which is constant in space and time, non-uniqueness may still occur if duality solutions are defined as a pair $(\rho, m)$. This is in contrast with the above lemma, since if $u$ is constant then $\rho u$ is necessarily a duality solution and so uniqueness holds for the triple $(\rho, m, \pi)$.
\end{remark}

\subsection*{Acknowledgement}
The work of C. P. is supported by the BOURGEONS and CRISIS projects, grants ANR-23-CE40-0014-01 and ANR-20-CE40-0020-01 of the French National Research Agency (ANR).
The research of NC and EZ  was supported by the EPSRC Early Career Fellowship  EP/V000586/1. Also, the work of NC was partly supported by the ``Excellence Initiative Research University (IDUB)" program at the University of Warsaw.

\bibliography{bib}
\end{document}